\documentclass[11pt]{amsart}
  
\usepackage[colorlinks=true, citecolor=black]{hyperref}
 \usepackage{comment}
\usepackage{latexsym}
\usepackage{graphicx}
\usepackage{amssymb}
 \graphicspath{{Figures/}}
 \usepackage[show]{ed}

\newcommand*{\rom}[1]{\expandafter\@slowromancap\romannumeral #1@}

\renewcommand{\epsilon}{\varepsilon}

\newcommand{\R}{{\mathbb R}}

\newcommand{\lan}{\left\langle}
\newcommand{\ran}{\right\rangle}
\newcommand{\mc}[1]{\mathcal{#1}}
\newcommand{\e}{\varepsilon}

\newcommand{\re}{\mathbb{R}}
\newcommand{\Cc}{C_c^\infty}
\newcommand{\Id}{\operatorname{Id}}

\newcommand{\vol}{{\operatorname{Vol}}}

\newcommand{\supp}{{\operatorname{supp\,}}}

\newcommand{\inj}{\operatorname{inj}}
\newtheorem{theo}{{\sc Theorem}}

\newtheorem{cor}{{\sc Corollary}}[section]

\newtheorem{lem}[cor]{{\sc Lemma}}

\numberwithin{equation}{section}

\newenvironment{rem}[1][]{\refstepcounter{cor}{\medskip\noindent{\it Remark~\thecor :\/#1} }}{\medskip}

\newcommand{\red}[1]{{{#1}}}

\usepackage{hyperref}

\title[Defect measures of eigenfunctions with maximal $L^\infty$ growth]{Defect measures of eigenfunctions with\\ maximal $L^\infty$ growth}

\author{Jeffrey Galkowski}
\address{Department of Mathematics, Stanford University, Stanford, CA USA}
\email{jeffrey.galkowski@stanford.edu }

\date{}

\begin{document}

\maketitle

\begin{abstract}
We characterize the defect measures of sequences of Laplace eigenfunctions with maximal $L^\infty$ growth. As a consequence, we obtain new proofs of results on the geometry of manifolds with maximal eigenfunction growth obtained by Sogge--Toth--Zelditch~\cite{SoggeTothZelditch}, and generalize those of Sogge--Zelditch~\cite{SZ16I} to the smooth setting. We also obtain explicit geometric dependence on the constant in H\"ormander's $L^\infty$ bound for high energy eigenfunctions, improving on estimates of Donnelly~\cite{Donnelly}.
\end{abstract}
\maketitle

\section{Introduction}

Let $(M,g)$ be a $C^{\infty}$ compact manifold of dimension $n$ without boundary. Consider the solutions to
\begin{equation}
\label{e:eigenfunction}
(-\Delta_g-\lambda_j^2)u_{\lambda_j}=0,\qquad \|u_{\lambda_j}\|_{L^2}=1
\end{equation}
as $\lambda_j\to\infty$. It is well known \cite{Ava,Lev,Ho68} (see also \cite[Chapter 7]{EZB}) that solutions to~\eqref{e:eigenfunction} satisfy
\begin{equation}
\label{e:L infinity}
\|u_{\lambda_j}\|_{L^\infty(M)}\leq C\lambda_j^{\frac{n-1}{2}}
\end{equation}
and that this bound is saturated e.g. on the sphere. Estimates for $L^p$ norms of eigenfunctions improving on those given by interpolation between~\eqref{e:L infinity} and $\|u_{\lambda_j}\|_{L^2}=1$ have been available since the seminal work of Sogge~\cite{So88}. \red{Since there are examples where these estimates are sharp,} it is natural to consider situations which produce sharp examples for~\eqref{e:L infinity}. 
Previous works \cite{Berard77,I-s,TZ02,SoggeZelditch,TZ03,SoggeTothZelditch,SZ16I,SZ16II} have studied the connections between growth of $L^\infty$ norms of eigenfunctions and the global geometry of the manifold $M$. The works of Sogge~\cite{So11} and Blaire--Sogge~\cite{BlSo15, BlSo17} study similar questions for low $L^p$ norms.

In this article, we study the relationship between $L^\infty$ growth and $L^2$ concentration of eigenfunctions (this direction of inquiry was initiated in~\cite{GT}). 
We measure $L^2$ concentration of eigenfunctions using {\em{defect measures}} - a sequence $\{u_{h_j}\}$ has defect measure $\mu$ if for any $a\in C_c^\infty(T^*M)$, 
\begin{equation} 
\label{defect}
\lan a(x,h_jD) u_{h_j},u_{h_j}\ran \to \int_{T^*M}  a(x,\xi)d\mu. 
\end{equation}
We write $a(x,hD)$ for a semiclassical pseudodifferential operator given by the quantization of the symbol $a(x,\xi)$ (see \cite[Chapters 4, 14]{EZB}) and let $h_j=\lambda_j^{-1}$ when considering the solutions to~\eqref{e:eigenfunction}.

By an elementary compactness/diagonalization argument it follows that any $L^2$ bounded sequence $u_{h}$ possesses  a further subsequence that has  a defect measure in the sense of (\ref{defect}) \cite[Theorem 5.2]{EZB}.  Moreover, a standard commutator argument shows that if 
$$
p(x,hD)u=o_{L^2}(h),
$$
for $p\in S^k(T^*M)$ real valued with
$$
|p|\geq c\langle \xi\rangle^k\text{ on }|\xi|\geq R,
$$
 then $\mu$ is supported on $\Sigma:=\{p=0\}$ and is invariant under the bicharacteristic flow  of $p;$ that is, if $G_t=\exp(tH_p): \Sigma \to \Sigma$ is the bicharacteristic flow,
$ (G_t)_*\mu =\mu,\,\forall t \in \R$ \cite[Theorems 5.3, 5.4]{EZB}. \par
\noindent
\red{
\begin{rem}
We will usually write $G_t(q)$ for the bicharacteristic flow applied to a point $q\in T^*M$. However, it will sometimes be useful distinguish between the position and momentum of $q$ and in these cases we will write $q=(x,\xi)$ and write $G_t(x,\xi)$ for the bicharacteristic flow applied to $(x,\xi)\in T^*M$. 
\end{rem}
}

Rather than studying only eigenfunctions of the Laplacian, we replace $-\Delta_g-\lambda_j^2$ by a general semiclassical pseudodifferential operator and replace eigenfunctions with quasimodes. To this end, we say that $u$ is compactly microlocalized if there exists $\chi \in \Cc(\re)$ with 
$$
Op_h(1-\chi(|\xi|))u=O_{\mc{S}}(h^\infty\|u\|_{L^2(M)})
$$
{where $Op_h$ is a quantization procedure giving pseuodifferential operators on $M$ (see e.g. \cite[Appendix E]{ZwScat}, see also Appendix~\ref{a:semiclassics}).}
For $P\in \Psi^m(M)$ {i.e. an $h$-pseudodifferential operator of order $m$}, we say that $u$ is a \emph{quasimode for $P$} if 
\begin{equation*}
\label{e:quasiDef}
Pu=o_{L^2}(h),\qquad \|u\|_{L^2}=1.
\end{equation*}
\begin{rem}
{Although $u$ implicitly depends on $h$, we suppress this in our notation to avoid overburdening the writing.}
\end{rem}

\red{For $x_0\in M$,} let $\Sigma_{\red{x_0}}:=\Sigma\cap T^*_{\red{x_0}}M$ and define respectively the \emph{flow out of $\Sigma_{\red{x_0}}$} and \emph{time $T$ flowout of $\Sigma_{\red{x_0}}$} by 
$$
\Lambda_{\red{x_0}}:=\bigcup_{T=0}^\infty \Lambda_{\red{x_0},T},\qquad \Lambda_{\red{x_0},T}:=\bigcup_{t=-T}^TG_t(\Sigma_{\red{x_0}}).
$$
{\begin{rem}
Note that in the case that $P=-h^2\Delta_g-1$, $\Sigma=S^*\!M$ and $\Sigma_{\red{x_0}}=S^*_{\red{x_0}}M.$ 
\end{rem}}

Let $\mc{H}^r$ denote the Hausdorff-$r$ measure with respect to the Sasaki metric on $T^*M$ {or more precisely the metric induced on $T^*M$ by pulling back the Sasaki metric on $TM$} (see for example \cite[Chapter 9]{BlairSasaki} for a treatment of the Sasaki metric). {Note that we choose to use the Sasaki metric on $T^*M$ induced by the metric on $M$ for concreteness, but any other metric on $T^*M$ will work equally well for our purposes.}  For a Borel measure $\rho$ on $T^*M$, let $\rho_{\red{x_0}}:=\rho|_{\Lambda_{\red{x_0}}}$ i.e. $\rho_{\red{x_0}}(A):=\rho(A\cap \Lambda_{\red{x_0}})$. Recall that two Borel measures on a set $\Omega$, $\mu$ and $\rho$, are \emph{mutually singular} (written $\mu\perp \rho$) if there exist disjoint sets $N,P\subset \Omega$ so that $\Omega=N\cup P$ and $\mu(N)=\rho(P)=0$. 

The main theorem characterizes the defect measures of quasimodes with maximal growth.

\begin{theo}
\label{mainthmCont}
Let $P\in \Psi^m(M)$ be an $h$-pseudodifferential operator with real principal symbol $p$ satisfying 
\begin{equation}
\label{e:pCond}
\partial_\xi p\neq 0\text{ on }\{p=0\}.
\end{equation}
Suppose $u$ is a compactly microlocalized quasimode for $P$  with
\begin{equation}
\label{e:quasiMax}
 \limsup_{h\to 0}h^{\frac{n-1}{2}}\|u\|_{L^\infty}>0
\end{equation}
and defect measure $\mu$.  Then there exists $\red{x_0}\in M$ and $x(h)\to \red{x_0}$ so that 
\begin{equation}
\label{e:decomp}
\limsup_{h\to 0}h^{\frac{n-1}{2}}|u(x(h))|>0,\qquad \mu_{\red{x_0}}= \rho_{\red{x_0}}+f d\mc{H}^n_{\red{x_0}},\qquad 
\end{equation}
where $0\neq f\in L^1(\Lambda_{\red{x_0}},\mc{H}^n_{\red{x_0}})$, $\rho_{\red{x_0}}\perp \mc{H}^n_{\red{x_0}}$, and both $f d\mc{H}^n_{\red{x_0}}$, and $\rho_{\red{x_0}}$ are invariant under $G_t$. 
\end{theo}

One way of interpreting Theorem~\ref{mainthmCont} is that a quasimode with maximal $L^\infty$ growth near $\red{x_0}$ must have energy on a positive measure set of directions entering $T^*_{\red{x_0}}M$. That is, it must have concentration comparable to that of the zonal harmonic. (See \cite[Section 4]{GT} for a description of the defect measure of the zonal harmonic.)

Theorem~\ref{mainthmCont} is an easy consequence of the following theorem (see section~\ref{tl-tg} for the proof that Theorem~\ref{thm:local} implies Theorem~\ref{mainthmCont}).

 \begin{theo} \label{thm:local}
 Let $\red{x_0}\in M$ and $P\in \Psi^m(M)$ be an $h$-pseudodifferential operator with real principal symbol $p$ satisfying 
$$
\partial_\xi p\neq 0\text{ on }\{p=0\}.
$$
There exists a constant $C_n$ depending only on $n$ with the following property:
Suppose that $u$ is compactly microlocalized quasimode for $P$
and has defect measure $\mu$. Define $\rho_{\red{x_0}}\perp \mc{H}^n_{\red{x_0}}$ and $f\in L^1(\Lambda_{\red{x_0}};\mc{H}^n_{\red{x_0}})$ by 
$$
\mu_{\red{x_0}}=: \rho_{\red{x_0}}+ fd\mc{H}^{n}_{\red{x_0}}.
$$  
Then for all $r(h)=o(1)$, 
$$
\limsup_{h\to 0}h^{\frac{n-1}{2}}\|u\|_{L^\infty(B(\red{x_0},r(h)){)}}\leq C_n\int_{\Sigma_{\red{x_0}}} \sqrt{f}\sqrt{\frac{|\nu(H_p)|}{|\partial_{\xi}p|_g}}d\vol_{\Sigma_{\red{x_0}}}
$$
where $\nu$ is a unit (with respect to the Sasaki metric) conormal to $\Sigma_{\red{x_0}}$ in $\Lambda_{\red{x_0}}$, $\vol_{\Sigma_{\red{x_0}}}$ is the measure induced by the \red{Sasaki} metric on $T^*\!M$, and $|\partial_\xi p|_g=|\partial_\xi p\cdot \partial_x|_g$. Furthermore, $fd\mc{H}^n_{\red{x_0}}$ is $G_t$ invariant.

In particular, if $\mu_{\red{x_0}}\perp \mc{H}^n_{\red{x_0}}$, then 
$$
\|u\|_{L^\infty(B(\red{x_0},r(h)){)}}=o(h^{\frac{1-n}{2}}).
$$
\end{theo}

\noindent \begin{rem}
\label{r:compact}
\begin{enumerate}
\item {We may assume without loss of generality that $\Sigma$ is compact. This follows from the fact that $u$ is compactly microlocalized. In particular, let $\chi\in C_c^\infty(\re)$ have $Op_h(1-\chi(|\xi|))u=O_{\mc{S}}(h^\infty)$. Then $u$ is a quasimode for $\tilde{P}=P+NOp_h(\langle \xi\rangle^m)Op_h(1-\chi(|\xi|))$ and for $N$ large enough, $\{\tilde{p}=0\}$ is compact. Therefore, we may work with $\tilde{p}$ rather than $p$. This furthermore implies that we may assume $\Sigma_{\red{x_0}}$ is a manifold since $\partial_\xi p\neq 0$ on $\Sigma$. }
\item {Note that $\partial_\xi p \cdot \partial_x = d\pi H_p$ where $\pi:T^*M\to M$ is the natural projection map. Therefore, $\partial_{\xi} p\cdot \partial_x$ is a well defined invariant vector field. The appearance of this factor in Theorem~\ref{thm:local} quantifies the fact that bicharacteristics of $H_p$ are not tangent to vertical fibers. It is precisely the tangency of these bicharacteristics which causes a change of behavior when $\partial_\xi p=0$. }
\item \red{Finally, observe that if one fixes geodesic normal coordinates at $x_0$, then the Sasaki metric on $T^*_{x_0}M$ is equal to the Euclidean metric and hence, in these coordinates, $d\vol_{\Sigma_{x_0}}$ is the volume induced by the Euclidean metric.}
\end{enumerate}
\end{rem}

To see that Theorem~\ref{thm:local} applies to solutions of~\eqref{e:eigenfunction}, let $h_j=\lambda_j^{-1}$. Writing $u=u_{\lambda_j}$ and $h=h_j$, 
$$
(-h^2\Delta_g-1)u=0.
$$
Then, $(-h^2\Delta_g-1)=p(x,hD)$ with $p=|\xi|_g^2-1+hr$ and therefore, the elliptic parametrix construction shows that $u$ is compactly microlocalized. Since $\partial_{\xi_j} p=2g^{ij}\xi_i,$ $\partial_\xi p\neq 0$ on $p=0$ and Theorem~\ref{thm:local} applies. In Section~\ref{tl-tg}, we use Theorem~\ref{thm:local} with $P=-h^2\Delta_g-1$ to give explicit bounds on the constant $C$ in \eqref{e:L infinity} in terms of the \emph{injectivity radius of $M$}, $\inj(M)$, thereby improving on the bounds of~\cite{Donnelly} at high energies.

\begin{cor}
\label{c:geomDep}
There exists $\tilde{C}_n>0$ depending only on $n$ so that for all $(M,g)$ compact, boundaryless Riemannian manifolds of dimension $n$ and all $\e>0$, there exists $\lambda_0=\lambda_0(\e,M,g)>0$ so that for $\lambda_j>\lambda_0$ and $u_{\lambda_j}$ solving~\eqref{e:eigenfunction}
$$
\|u_{\lambda_j}\|_{L^\infty}\leq\Big(\frac{\tilde{C}_n}{\inj(M)^{1/2}}+\e\Big)\lambda_j^{\frac{n-1}{2}}.
$$
\end{cor}

Theorem~\ref{thm:local} is sharp in the following sense. Let $P=-h^2\Delta_g-1$ and $G_t$ as above.

\begin{theo}
\label{thm:modes}
Suppose there exists $\red{z_0}\in M$, $T>0$ so that $G_T(\red{z_0},\xi)=(\red{z_0},\xi)$ for all $(\red{z_0},\xi)\in S^*_{\red{z_0}}M$.  Let $\rho_{\red{z_0}}\perp \mc{H}^n_{\red{z_0}}$ be a Radon measure on $\Lambda_{\red{z_0}}$ invariant under $G_t$ and $0\leq f\in L^1(\Lambda_{\red{z_0}}, \mc{H}^n_{\red{z_0}})$ be invariant under $G_t$ so that 
$$
\|f\|_{L^1(\Lambda_{\red{z_0}}, \mc{H}^n_{\red{z_0}})}+\rho_{\red{z_0}}(\Lambda_{\red{z_0}})=1.
$$ Then there exist $h_j\to 0$ and $\{u_{h_j}\}_{j=1}^\infty$ solving 
$$
(-h_j^2\Delta_g-1)u_{h_j}=o(h_j),\qquad \|u_{h_j}\|_{L^2}=1,\qquad  \limsup_{j\to \infty}h_j^{\frac{n-1}{2}}\|u_{h_j}\|_{L^\infty}\geq  (2\pi )^{\frac{1-n}{2}}\int_{\Sigma_{\red{z_0}}}\sqrt{f}d\textup{Vol}_{\Sigma_{\red{z_0}}}
$$
and having defect measure $\mu=\rho_{\red{z_0}} +fd\vol_{\Lambda_{\red{z_0}}}$.
\end{theo}

Notice that we do not claim the existence of exact eigenfunctions having prescribed defect measures in Theorem~\ref{thm:modes}, instead constructing only quasimodes.

\subsection{Relation with previous results}

As far as the author is aware, the only previous work giving conditions on the defect measures of eigenfunctions with maximal $L^\infty$ growth is \cite{GT}. Theorem~\ref{thm:local} improves on the conditions given in \cite[Theorem 3]{GT}; replacing $\mc{H}_{\red{x_0}}^n(\supp \mu_{\red{x_0}})=0$ with the sharp condition $\mu_{\red{x_0}}\perp \mc{H}^n_{\red{x_0}}$. To see an example of how these conditions differ, fix $\red{x_0}\in M$ such that every geodesic through $\red{x_0}$ is closed and let $\{\xi_k\}_{k=1}^\infty\subset S^*_{\red{x_0}}M$ be a countable dense subset. Suppose that the defect measure of $\{u_{\lambda_j}\}$ is given by 
$$
\mu=\sum_k a_k\delta_{\gamma_k},\qquad a_k>0
$$
where $\gamma_k$ is the geodesic emanat{ing} from $(\red{x_0},\xi_k)$. Then $\supp \mu_{\red{x_0}}=\Lambda_{\red{x_0}}$, but $\mu_{\red{x_0}}\perp \mc{H}^n_{\red{x_0}}$, so Theorem~\ref{thm:local} applies to this sequence but the results of \cite{GT} do not. Furthermore, Theorem~\ref{thm:local} gives quantitative estimates on the growth rates of quasimodes in terms of their defect measures.

We are able to draw substantial conclusions about the global geometry of a manifold $M$ having quasimodes with maximal $L^\infty$ growth from Theorem~\ref{thm:local}. The results of Sogge--Toth--Zelditch \cite[Theorems 1(1), 2]{SoggeTothZelditch} and hence also Sogge--Zelditch \cite[Theorem 1.1]{SoggeZelditch} are corollaries of Thoerem~\ref{thm:local}. For $\red{x_0}\in M$, define the map $T_{\red{x_0}}:\Sigma_{\red{x_0}}\to \re\sqcup\{\infty\}$ by 
\begin{equation}
\label{e:returnTime}
T_{\red{x_0}}(\xi):=\inf\{t>0\mid G_t(\red{x_0},\xi)\in \Sigma_{\red{x_0}}\}.
\end{equation}
Then, define the \emph{loop set} by
$$
\mc{L}_{\red{x_0}}:=\{ \xi\in \Sigma_{\red{x_0}} \mid T_{\red{x_0}}(\xi)<\infty\},
$$
and the \emph{first return map} $\eta_{\red{x_0}}:\mc{L}_{\red{x_0}}\to \Sigma_{\red{x_0}}$ by
$$
G_{T_{\red{x_0}}(\xi)}(\red{x_0},\xi)=(\red{x_0},\eta_{\red{x_0}}(\xi)).
$$
Finally, define the set of \emph{recurrent points} by
\begin{equation}
\label{e:recurDef}
\mc{R}_{\red{x_0}}:=\left\{\xi\in \Sigma_{\red{x_0}}\mid \xi\in \left(\bigcap_{T>0}\overline{\bigcup_{ t>T}G_t(\red{x_0},\xi)\cap \Sigma_{\red{x_0}}}\right)\bigcap\left(\bigcap_{T>0}\overline{\bigcup_{ t>T}G_{-t}(\red{x_0},\xi)\cap \Sigma_{\red{x_0}}}\right)\right\},
\end{equation}
where the closure is with respect to the subspace topology on $\Sigma_{\red{x_0}}$. 

\begin{cor}
\label{c:recur}
Let $(M,g)$ be a compact boundaryless Riemannian manifold and $P$ satisfy~\eqref{e:pCond}. Suppose that $\vol_{\Sigma_{\red{x_0}}}(\mc{R}_{\red{x_0}})=0$. Then for any $r(h)=o(1)$ and $u$ a compactly microlocalized quasimode for $P$,  
$$
\|u\|_{L^\infty(B(\red{x_0},r(h)){)}}=o(h^{\frac{1-n}{2}}).
$$
\end{cor}

Moreover, the forward direction of \cite[Theorem 1.1]{SZ16I} with the analyticity assumption removed is an easy corollary of Theorem~\ref{thm:local}. To state the theorem recall that $d\vol_{\Sigma_{\red{x_0}}}$ denote the measure induced on $\Sigma_{\red{x_0}}$ from the \red{Sasaki} metric on $T^*M$. We define the unitary Perron--Frobenius operator $U_{\red{x_0}}:L^2(\mc{R}_{\red{x_0}}, d\vol_{\Sigma_{\red{x_0}}})\to L^2(\mc{R}_{\red{x_0}},d\vol_{\Sigma_{\red{x_0}}})$ by
\begin{equation}
\label{e:pF}
U_{\red{x_0}}(f)(\xi):= \sqrt{J_{\red{x_0}}(\xi)}f(\eta_{\red{x_0}}(\xi)),
\end{equation}
where, writing 
$$
G_t(\red{x_0},\xi)=(x_t(\red{x_0},\xi),\eta_t(\red{x_0},\xi)),
$$
we have that
\begin{equation}
\label{e:jacobian}
J_{\red{x_0}}(\xi)=\big|\det D_\xi \eta_t|_{t=T_{\red{x_0}}(\xi)}\big|
\end{equation}
 is the Jacobian factor so that for $f\in L^1(\Sigma_{\red{x_0}})$ supported on $\mc{L}_{\red{x_0}}$, 
$$
\int \eta_{\red{x_0}}^*f J_{\red{x_0}}(\xi)d\vol_{\Sigma_{\red{x_0}}}=\int f(\xi)d\vol_{\Sigma_{\red{x_0}}}.
$$ 
See \cite[Section 4]{Saf88} for a more detailed discussion of $U_{\red{x_0}}$. We say that $\red{x_0}$ is \emph{dissipative} if 
\begin{equation}
\label{e:dissipative}
\Big\{f\in L^2\Big(\mc{R}_{\red{x_0}},d\vol_{\Sigma_{\red{x_0}}}\Big)\,\Big|\, U_{\red{x_0}}(f)=f\Big\}=\{0\}.
\end{equation}
\begin{cor}
\label{c:invariant}
Let $(M,g)$ be a compact boundaryless Riemannian manifold and $P$ satisfy~\eqref{e:pCond}. Suppose that $\red{x_0}$ is dissipative. Then for $r(h)=o(1)$ and $u$ a compactly microlocalized quasimode for $P$,
$$
\|u\|_{L^\infty(B(\red{x_0},r(h)))}=o(h^{\frac{1-n}{2}}).
$$
\end{cor}

The dynamical arguments in \cite{SZ16II} show that if $(M,g)$ is a real analytic surface and $P=-h^2\Delta_g-1$, then $\red{x_0}$ being non-dissipative implies that $\red{x_0}$ is a periodic point for the geodesic flow, i.e. a point so that there is a $T>0$ so that every geodesic starting from $(\red{x_0},\xi)\in S^*_{\red{x_0}}M$ smoothly closes at time $T$. 

\subsection{Comments on the proof}

While the assumption $Pu=o_{L^2}(h)$ implies a global assumption on $u$, similar to that in \cite{GT}, the analysis here is entirely local. The global consequences in Corollaries~\ref{c:recur} and~\ref{c:invariant} follow from dynamical arguments using invariance of defect measures. 

We take a different approach from that in \cite{GT} choosing to base our method on the Koch--Tataru--Zworski method \cite{KTZ} rather than explicit knowledge of the spectral projector. This approach gives a more explicit explanation for the $L^\infty$ improvements from defect measures. In Section~\ref{s:laplace} we sketch the proof of Theorem~\ref{thm:local} in the case that $\mu_{\red{x_0}}\perp \mc{H}^n_{\red{x_0}}$ using the spectral projector.

The idea behind our proof is to estimate the absolute value of $u$ at $\red{x_0}$ in terms of the degree to which energy concentrates along each bicharacteristic passing through $\Sigma_{\red{x_0}}$. Either too much localization or too little localization will yield an improvement over the naive bound. By covering $\Lambda_{\red{x_0}}$ with appropriate cutoffs to tubes around bicharacteristics we are then able to give $o(h^{\frac{1-n}{2}})$ bounds whenever $\mu_{\red{x_0}}\perp \mc{H}^n_{\red{x_0}}$.  The proof relies, roughly, on the fact that if a compactly microlocalized function $u$ on $\R^m$ has defect measure supported at $(\red{y}_0,\red{\eta}_0)$, then $\|u\|_{L^\infty}=o(h^{-m/2})$ rather than the standard estimate $O(h^{-m/2}).$
\\

\noindent{\sc Acknowledgements} The author would like to thank John Toth for many stimulating discussions and for comments on a previous version. The author is grateful to the referee for careful reading and many helpful comments which improved the exposition. Thanks also to the National Science Foundation for support under the Mathematical Sciences Postdoctoral Research Fellowship  DMS-1502661.

\section{Consequences of Theorem~\ref{thm:local}}
\label{tl-tg}

We first formulate a local result matching those in \cite{SoggeZelditch, SoggeTothZelditch} more closely.

\begin{cor}
\label{c:lessLocal}
Let $\red{x_0}\in M$ and $P\in \Psi^m(M)$ satisfying the assumption of Theorem \ref{thm:local}. 
Then there exists a constant $C_n$ depending only on $n$ with the following property. Suppose that $u$ is a compactly microlocalized quasimode for $P$, and has defect measure $\mu$. Define $\rho_{\red{x_0}}\perp \mc{H}^n_{\red{x_0}}$ and $f\in L^1(\Lambda_{\red{x_0}};\mc{H}^n_{\red{x_0}})$ by
$$
\mu_{\red{x_0}}=:\rho_{\red{x_0}} +fd\mc{H}^n_{\red{x_0}}.
$$ 
Then for all $\e>0$, there exists a neighborhood $\mc{N}(\e)$ of $\red{x_0}$ and $h_0(\e)$ such that for $0<h<h_0(\e)$,
$$
\|u\|_{L^\infty(\mc{N}(\e))}\leq h^{-\frac{n-1}{2}}\left(C_n\int_{\Sigma_{\red{x_0}}}\sqrt{f}\sqrt{\frac{|\nu(H_p)|}{|\partial_\xi p |_g}}d\vol_{\Sigma_{\red{x_0}}}+\e\right).
$$
\end{cor}

\begin{proof}[Proof that Theorem~\ref{thm:local} implies Corollary~\ref{c:lessLocal}]
Let 
$$
\tilde{A}_{\red{x_0}}:=C_n\int_{\Sigma_{\red{x_0}}} \sqrt{f}\sqrt{\frac{\nu(H_p)}{|\partial_\xi p|_g}}d\vol_{\Sigma_{\red{x_0}}}
$$
and suppose that there exists $\e>0$ such that for all $r>0$, 
\begin{equation}
\label{e:blah}
\limsup_{h\to 0}h^{\frac{1-n}{2}}\|u_h\|_{L^\infty(B(\red{x_0},r))}>\tilde{A}_{\red{x_0}}+\e.
\end{equation}

Fix $r_0>0$. Then by~\eqref{e:blah} there exists $x\in B(\red{x_0},r_0)$, $h_0>0$ so that 
$$
|u_{h_0}(x)|h_0^{\frac{n-1}{2}}\geq \tilde{A}_{\red{x_0}} + \frac{\e}{2}.
$$ 
Assume that there exist $\{h_j\}_{j=0}^N$ and $\{x_j\}_{j=0}^N$ so that 
$$
h_j\leq \frac{h_{j-1}}{2},\qquad x_j \in B(\red{x_0},r_02^{-j}),\qquad h_j^{\frac{n-1}{2}}|u(x_j)|\geq \tilde{A}_{\red{x_0}}+ \frac{\e}{2}.
$$

By~\eqref{e:blah}, there exists $h_k\downarrow 0$ and $x_k\in B(\red{x_0},r_02^{-N-1})$ such that 
$$
h_k^{\frac{1-n}{2}} |u_{h_k}(x_k)|\geq \tilde{A}_{\red{x_0}} + \frac{\e}{2}.
$$
Therefore, we can choose $k_0$ large enough so that $h_{k_0}\leq \frac{h_N}{2}$ and let $(h_{N+1},x_{N+1})=(h_{k_0},x_{k_0})$, Hence, by induction, there exists 
$h_j\downarrow 0$, $x_j\to \red{x_0}$ such that 
$$
h_{j}^{\frac{n-1}{2}}|u_{h_j}(x_j)|\geq \tilde{A}_{\red{x_0}} + \frac{\e}{2},
$$
contradicting Theorem~\ref{thm:local}.
\end{proof}

\begin{proof}[Proof that Theorem~\ref{thm:local} implies Theorem~\ref{mainthmCont}]
Compactness of $M$ together with Corollary~\ref{c:lessLocal} with $f\equiv 0$ implies the contrapositive of Theorem~\ref{mainthmCont}, in particular, if $\mu_{\red{x_0}}\perp \mc{H}^n_{\red{x_0}}$ for all $\red{x_0}$, then $\|u\|_{L^\infty}=o(h^{\frac{1-n}{2}}).$
\end{proof}

\subsection{Proof of Corollaries~\ref{c:recur} and \ref{c:invariant} from Theorem~\ref{thm:local}}

\begin{lem}
\label{l:suppRecur}
Fix $\red{x_0}\in M$ and suppose that $u$ is compactly microlocalized with $Pu=o_{L^2}(h)$. Define $\rho_{\red{x_0}}\perp \mc{H}^n_{\red{x_0}}$ and $f\in L^1(\Lambda_{\red{x_0}};\mc{H}^n_{\red{x_0}})$ by 
$$
\mu_{\red{x_0}}=\rho_{\red{x_0}}+fd\mc{H}^n_{\red{x_0}}.
$$ 
Then {$f|_{\Sigma_{\red{x_0}}}\in L^1(\vol_{\Sigma_{\red{x_0}}})$ and $ f|_{\Sigma_{\red{x_0}}}(1-1_{\mc{R}_{\red{x_0}}})=0$ almost everywhere.}
\end{lem}

\begin{proof}
For $\xi_\red{0}\in \Sigma_{\red{x_0}}$ and $\e>0$ let $B(\xi_{\red{0}},\e)\subset \Sigma_{\red{x_0}}$ be the open ball of radius $\e$ and 
$$
V:=\bigcup_{-2\delta<t<2\delta} G_t(B(\xi_{\red{0}},\e)).
$$
Observe that by Theorem~\ref{thm:local} the triple $(\Lambda_{\red{x_0}}, fd\mc{H}^n_{\red{x_0}},G_t)$ forms a measure preserving dynamical system. The Poincar\'e recurrence  theorem \cite[Proposition 4.2.1, 4.2.2]{BrinStuck} implies that for $fd\mc{H}^n_{\red{x_0}}$ a.e. $(x,\xi)\in V$ there exists $t^{\pm}_n\to \pm\infty$ so that $G_{t^{\pm}_n}(x,\xi)\in V$. By the definition of $V$, there exists $s^{\pm}_n$ with $|s^{\pm}_n-t^{\pm}_n|<2\delta$ such that $G_{s^{\pm}_n}(x,\xi)\in B(\xi_{\red{0}},\e)$. In particular, for $fd\mc{H}^n_{\red{x_0}}$ a.e. $(x,\xi)\in V$, 
\begin{equation}
\label{e:recurs}
\bigcap_{T>0}\overline{\bigcup_{t\geq T}G_t(x,\xi)\cap B(\xi_{\red{0}},\e)}\neq \emptyset,\qquad\bigcap_{T>0}\overline{\bigcup_{t\geq T}G_{-t}(x,\xi)\cap B(\xi_{\red{0}},\e)}\neq \emptyset.
\end{equation}

Let 
$$
\mu_{\Sigma_{\red{x_0}}}:=f|_{\Sigma_{\red{x_0}}}|\nu(H_p)||_{\Sigma_{\red{x_0}}}d\vol_{\Sigma_{\red{x_0}}}.
$$
We next show that~\eqref{e:recurs} holds for $\mu_{\Sigma_{\red{x_0}}}$ a.e. point in $B(\xi_{\red{0}},\e)$.  To do so, suppose the opposite. Then there exists $A\subset B(\xi_{\red{0}},\e)$ with $\mu_{\Sigma_{\red{x_0}}}(A)>0$ so that for each $(x,\xi)\in A$, there exists $T>0$ with
\begin{equation}
\label{e:norecur}
\left(\left[\bigcup_{t\geq T}G_{t}(x,\xi)\right]\bigcup\left[ \bigcup_{t\geq T}G_{-t}(x,\xi)\right]\right)\bigcap B(\xi_{\red{0}},\e)=\emptyset.
\end{equation}
Let 
$$
A_\delta:=\bigcup_{t=-\delta}^\delta G_t(A).
$$
Then $A_\delta\subset V$ and for all $(x,\xi)\in A_\delta$, there exists $T>0$ so that \eqref{e:norecur} holds.
Moreover, invariance of $fd\mc{H}^n_{\red{x_0}}$ under $G_t$ together with  Lemma~\ref{l:inv} implies that
$$
(fd\mc{H}^n_{\red{x_0}})(A_\delta)=2\delta \mu_{\Sigma_{\red{x_0}}}(A)>0
$$
which contradicts~\eqref{e:recurs}. Thus~\eqref{e:recurs} holds for $\mu_{\Sigma_{\red{x_0}}}$ a.e. point in $B(\xi_{\red{0}},\e)$.  

Let $\{B(\xi_i,\e_i)\}$ be a countable basis for the topology on $\Sigma_{\red{x_0}}$. Then for each $i$, there is a subset of full measure, $\tilde{B}_i\subset B(\xi_i,\e_i)$ so that for every point of $\tilde{B}_i$ \eqref{e:recurs} holds with $\xi_{\red{0}}=\xi_i$, $\e=\e_i$. Noting that $X_i=\tilde{B}_i\cup (\Sigma_{\red{x_0}}\setminus B(\xi_i,\e_i))$ has full measure, we conclude that $\tilde{\Sigma}_{\red{x_0}}=\cap_i X_i\subset \mc{R}_{\red{x_0}}$ has full measure and thus, $\mu_{\Sigma_{\red{x_0}}}(\mc{R}_{\red{x_0}})=\mu_{\Sigma_{\red{x_0}}}(\Sigma_{\red{x_0}})$, finishing the proof of the lemma.
\end{proof}

\begin{proof}[Proof of Corollary~\ref{c:recur}]
Let $u$ solve $Pu=o_{L^2}(h)$. Then we can extract a subsequence with a defect measure $\mu$. By Lemma~\ref{l:suppRecur}, $\mu_{\red{x_0}}=\rho_{\red{x_0}}+ fd\mc{H}^n_{\red{x_0}}$ with $\rho_{\red{x_0}}\perp \mc{H}^n_{\red{x_0}}$ and $\supp f|_{\Sigma_{\red{x_0}}}\subset \mc{R}_{\red{x_0}}$. Now, if $\vol_{\Sigma_{\red{x_0}}}(\mc{R}_{\red{x_0}})=0$, 
$$
\int_{\Sigma_{\red{x_0}}}\sqrt{f} d\vol_{\Sigma_{\red{x_0}}}=0.
$$
Plugging this into Theorem~\ref{thm:local} proves the corollary{.}
\end{proof}

\begin{proof}[Proof of Corollary~\ref{c:invariant}]
Let $u$ solve $Pu=o_{L^2}(h)$. Then we can extract a subsequence with a defect measure $\mu$. By Lemma~\ref{l:suppRecur} and Theorem~\ref{thm:local}, $\mu_{\red{x_0}}=\rho_{\red{x_0}}+ fd\mc{H}^n_{\red{x_0}}$ where $\rho_{\red{x_0}}\perp \mc{H}^n_{\red{x_0}}$, $\supp f|_{\Sigma_{\red{x_0}}}\subset \mc{R}_{\red{x_0}}$, and $fd\mc{H}^n_{\red{x_0}}$ is $G_t$ invariant.

Let $T_{\red{x_0}}$ be as in~\eqref{e:returnTime}. Fix $T<\infty$ and suppose 
$$
A\subset \Omega_T:=\{\eta\in \Sigma_{\red{x_0}}\mid T_{\red{x_0}}(\eta)\leq T\}.
$$ 
Write $(0,T]=\bigsqcup_{i=1}^{N(\e)}(T_i-\e,T_i+\e]$ and 
$$
\Omega_T=\bigsqcup_{i=1}^{N(\e)}{\Omega^\e_i},\qquad {\Omega^\e_i:=}T_{\red{x_0}}^{-1}( (T_i-\e,T_i+\e]).
$$
Then, by Lemma~\ref{l:inv} {(using that in the case of $-h^2\Delta_g-1$, $|\nu(H_p)|\equiv 2$)} for any $0<\delta$ small enough
\begin{align*}
\int {2}1_A fd\vol_{\Sigma_{\red{x_0}}}&= \frac{1}{2\delta}\int  1_{\bigcup_{-\delta}^\delta G_t(A)}fd\mc{H}^n_{\red{x_0}}\,{=}\,\frac{1}{2\delta}\sum_i\int  1_{\bigcup_{-\delta}^\delta G_t(A\cap \Omega^\e_i)}fd\mc{H}^n_{\red{x_0}}.
\end{align*}
{Next, using invariance of $fd\mc{H}^n_{\red{x_0}}$ under $G_t$, we have}
\begin{align*}
\frac{1}{2\delta}\sum_i\int  1_{\bigcup_{-\delta}^\delta G_t(A\cap \Omega^\e_i)}fd\mc{H}^n_{\red{x_0}}&=\sum_i \frac{1}{2\delta}\int 1_{\bigcup_{{T}_i-\delta}^{{T}_i+\delta} G_t(A\cap \Omega^\e_i)}f d\mc{H}^n_{\red{x_0}}
\end{align*}
{Then, by the definition of $\Omega^\e_i$, for $q\in \Omega^\e_i$, $|T_{\red{x_0}}(q)-T_i|<\e$ and
$$
\sum_i1_{\bigcup_{{T}_i-\delta}^{{T}_i+\delta} G_t(A\cap \Omega^\e_i)}\underset{\e \to 0}{\longrightarrow} 1_{\bigcup_{-\delta}^\delta G_t(\eta_x(A))} \qquad fd\mc{H}^n_{\red{x_0}} \text{ a.e.}
$$
 In particular, by the dominated convergence theorem}
\begin{align*}
\lim_{\e\to 0}\sum_i \frac{1}{2\delta}\int 1_{\bigcup_{{T}_i-\delta}^{{T}_i+\delta} G_t(A\cap \Omega^{\red{\e}}_i)}f d\mc{H}^n_{\red{x_0}}
&=\frac{1}{2\delta}\int 1_{\bigcup_{-\delta}^{\delta} G_t(\eta_x(A))}f d\mc{H}^n_{\red{x_0}}
\end{align*}
So, sending $\delta\to 0$ gives
$$
{2}\int 1_A fd\vol_{\Sigma_{\red{x_0}}}={2}\int 1_{\eta_x(A)} fd\vol_{\Sigma_{\red{x_0}}}
$$
for all $A\subset \Omega_T$ measurable. Taking $T\to \infty$ then proves this for all $A\subset \mc{L}_{\red{x_0}}$ measurable. In particular, changing variables, using that $\supp f\subset \mc{R}_{\red{x_0}}\subset \mc{L}_{\red{x_0}}$, and writing $J_{\red{x_0}}(\xi)$ as in~\eqref{e:jacobian}
$$
f(\xi)d\vol_{\Sigma_{\red{x_0}}}(\xi)=f(\eta_{\red{x_0}}(\xi))\cdot J_ {\red{x_0}}(\xi)d\vol_{\Sigma_{\red{x_0}}}(\xi)
$$
which implies $U_{\red{x_0}} \sqrt{f}=\sqrt{f}$ where $U_{\red{x_0}}$ is defined in~\eqref{e:pF}. Observe that since $\red{x_0}$ is dissipative and $\sqrt{f}\in L^2(\mc{R}_{\red{x_0}},d\vol_{\Sigma_{\red{x_0}}})$, \eqref{e:dissipative} implies $\sqrt{f}=0$. Theorem~\ref{thm:local} then completes the proof.
\end{proof}

\subsection{Spectral cluster estimates for $-\Delta_g$}
Let $(M,g)$ be a smooth, compact, boundaryless Riemannian manifold of dimension $n$, $p=|\xi|_g^2-1$, $G_t=\exp(tH_p)$ and
$$
A_x:=\frac{C_n}{2}\left(\frac{\vol_{\Sigma_{\red{x_0}}}(\mc{R}_{\red{x_0}})}{\inf_{\xi\in \mc{R}_{\red{x_0}}}T_{\red{x_0}}(\xi)}\right)^{1/2}
$$
where $T_{\red{x_0}}$ is as in~\eqref{e:returnTime} and $C_n$ is the constant in Theorem~\ref{thm:local}. We consider an orthonormal basis $\{u_{\lambda_j}\}_{j=1}^\infty$ of eigenfunctions of $-\Delta_g$ (i.e. solving~\eqref{e:eigenfunction}) and let 
$$
\Pi_{[\lambda,\lambda+\delta]}:=1_{[\lambda,\lambda+\delta]}(\sqrt{-\Delta_g}).
$$

\begin{cor}
\label{l:volcano}
For all $\e>0$, $\red{x_0}\in M$, there exists $\delta=\delta(\red{x_0},\e)>0$, a neighborhood $\mc{N}(\red{x_0},\e)$ of $\red{x_0}$, and $\lambda_0=\lambda_0(\red{x_0},\e)>0$ so that for $\lambda>\lambda_0$, 
\begin{equation}
\label{e:volcano}
\|\Pi_{[\lambda,\lambda+\delta]}\|_{L^2{(M)}\to L^\infty(\mc{N}(\red{x_0},\e))}^2=\sup_{y\in \mc{N}(\red{x_0},\e)}\sum_{\lambda_j\in[\lambda,\lambda+\delta]}|u_{\lambda_j}(y)|^2\leq (A_{\red{x_0}}^2+\e)\lambda^{n-1}.
\end{equation}
\end{cor}

Note that since $G_t|_{S^*M}$ parametrizes the speed 2 geodesic flow and therefore
\begin{gather*}
\inf_{\xi\in \mc{R}_{\red{x_0}}}T_{\red{x_0}}(\xi)\geq \frac{1}{2}L(\red{x_0},M)\geq \inj(M),\\
L(\red{x_0},M):=\inf\{t>0\mid \text{ there exists a geodesic of length }t\text{ starting and ending at }\red{x_0}\},
\end{gather*}
and $\inj(M)$ denotes the injectivity radius of $M$.
Therefore, we could replace $A_{\red{x_0}}$ in~\eqref{e:volcano} by either of
$$
A'_{\red{x_0}}=C_n\left(\frac{\vol_{\Sigma_{\red{x_0}}}(\mc{R}_{\red{x_0}})}{2\cdot L(\red{x_0},M)}\right)^{1/2},\qquad A''_{\red{x_0}}=C_n\left(\frac{\vol_{\Sigma_{\red{x_0}}}(\mc{R}_{\red{x_0}})}{4\cdot \inj(M)}\right)^{1/2}.
$$
to obtain a weaker, but more easily understood statement. Corollary~\ref{l:volcano} is closely related to the work of Donnelly~\cite{Donnelly} and gives explicit dependence of the constant in the H\"ormander bound in terms of geometric quantities.

\begin{proof}
{We start from the fact that for $U\subset M$}
\begin{equation}
\label{e:squid}
\|\Pi_{[\lambda,\lambda+\delta]}\|_{L^2(M)\to L^\infty(U)}^2=\sup_{x\in U}\sum_{\lambda_j\in [\lambda,\lambda+\delta]}|u_{\lambda_j}(x)|^2.
\end{equation}

For $w\in L^2(M)$,
\begin{equation}
\label{e:quasi1}
\|(-\Delta_g-\lambda^2)\Pi_{[\lambda,\lambda+\delta]}w\|_{L^2}\leq 2\lambda\delta \|\Pi_{[\lambda,\lambda+\delta]}w\|_{L^2}.
\end{equation}

Suppose that for some $\e>0$ no $\delta$, $\mc{N}(\red{x_0})$, and $\lambda_0$ exist so that~\eqref{e:volcano} holds. Then for all $\delta>0$, $r>0$, 
$$
\limsup_{\lambda \to \infty}\lambda^{\frac{1-n}{2}}\|\Pi_{[\lambda,\lambda+\delta]}\|_{L^2(M)\to L^\infty(B(\red{x_0},r))}> A_{\red{x_0}}+\e.
$$
Therefore, for all ${0<m\in \mathbb{Z}}$, there exists $\lambda_{k,{m}}\uparrow \infty$ so that 
\begin{equation}
\label{e:gecko}
\lambda_{k,{m}}^{\frac{1-n}{2}}\|\Pi_{[\lambda_{k,{m}},\lambda_{k,{m}}+{m^{-1}}]}\|_{L^2(M)\to L^\infty(B(\red{x_0},r))}> A_{\red{x_0}}+\e.
\end{equation}
Moreover, we may assume that for ${m_1<m_2}$, $\lambda_{k,{m_2}}>\lambda_{k,{m_1}}$. {Indeed, assume we have chosen such $\lambda_{k,m}$ for $m<M$. Then there exists $\lambda_{k,M}> \max(\lambda_{k,M-1},\lambda_{k-1,M})$ so that~\eqref{e:gecko} holds with $m=M$. By convention, we let $\lambda_{-1,m}=0$.}   Now, for ${m_1\leq m_2}$,
$$
\|\Pi_{[\lambda,\lambda+{m_2^{-1}}]}\|_{L^2(M)\to L^\infty(B(\red{x_0},r))}\leq \|\Pi_{[\lambda,\lambda+{m_1^{-1}}]}\|_{L^2(M)\to L^\infty(B(\red{x_0},r))},
$$
letting $\lambda_l=\lambda_{l,{l}}$, $\lambda_l\to \infty$ and 
$$
\lambda_l^{\frac{1-n}{2}}\|\Pi_{[\lambda_{l},\lambda_{l}+l^{-1}]}\|_{L^2(M)\to L^\infty(B(\red{x_0},r))}> A_{\red{x_0}}+\e.
$$

By~\eqref{e:quasi1} for $w\in L^2(M)$
$$
\|(-\lambda_l^{-2}\Delta_g-1)\Pi_{[\lambda_l,\lambda_l+l^{-1}]}w\|_{L^2\to L^2}=o(\lambda_l^{-1})\|\Pi_{[\lambda_l,\lambda_l+l^{-1}]}w\|_{L^2\to L^2}.
$$
Fix $w_l\in L^2(M)$ with $\|w_l\|_{L^2}=1$, so that 
$$
\lambda_l^{\frac{1-n}{2}}\|v_l\|_{L^\infty(B(\red{x_0},r))}> A_{\red{x_0}}+\e,\qquad v_l:=\Pi_{[\lambda_l,\lambda_l+l^{-1}]}w_l.
$$
Then extracting a further subsequence if necessary, we may assume that $v_l$ has defect measure $\mu$ with $\mu_{\red{x_0}}=\rho_{\red{x_0}}+fd\mc{H}^n_{\red{x_0}}$ and hence that Corollary~\ref{c:lessLocal} applies to $v_l$. Furthermore, since $\|v_l\|_{L^2}\leq \|w_l\|_{L^2}=1$, 
\begin{equation}
\label{e:spider}
\int_{\Lambda_{\red{x_0}}} fd\mc{H}^n_{\red{x_0}}\leq 1.
\end{equation}

By computing in normal geodesic coordinates at $\red{x_0}$, observe that for $p=|\xi|^2_g-1$, $|\nu(H_p)|=|\partial_{\xi} p|_g=2$. Thus, Corollary~\ref{c:lessLocal}, implies the existence of $r>0$ small enough so that 
\begin{align}
\label{e:ant}
A_{\red{x_0}}+\e\leq \limsup_{l\to \infty}\lambda_l^{\frac{1-n}{2}}\|v_l\|_{L^\infty(B(\red{x_0},r))}&\leq  C_n\int_{\Sigma_{\red{x_0}}}\sqrt{f}d\vol_{\Sigma_{\red{x_0}}}
\end{align} 
Finally, by Lemma~\ref{l:suppRecur} and \eqref{e:spider}, $\supp f\subset \mc{R}_{\red{x_0}}$ and $\|f\|_{L^1(\Lambda_{\red{x_0}},\mc{H}^n_{\red{x_0}})}\leq 1$. Therefore,
\begin{align*}
C_n\int_{\Sigma_{\red{x_0}}}\sqrt{f}d\vol_{\Sigma_{\red{x_0}}}&\leq C_n\left( \frac{1}{2}\int_{\Sigma_{\red{x_0}}}f|\nu(H_p)|d\vol_{\Sigma_{\red{x_0}}}\right)^{1/2}\big(\vol_{\Sigma_{\red{x_0}}}(\mc{R}_{\red{x_0}}\big)\big)^{1/2}\\
&= C_n\left( \frac{1}{4\cdot\inf_{\xi\in \mc{R}_{\red{x_0}}}(T_{\red{x_0}}(\xi))}\int_{\Lambda_{\red{x_0},\inf_{\mc{R}_{\red{x_0}}}T_{\red{x_0}}(\xi)}}fd\mc{H}^n_{\red{x_0}}\right)^{1/2}\big(\vol_{\Sigma_{\red{x_0}}}(\mc{R}_{\red{x_0}}\big)\big)^{1/2}\\
&\leq \frac{C_n}{2}\left(\frac{{\vol_{\Sigma_{\red{x_0}}}(\mc{R}_{\red{x_0}})}}{\inf_{\xi\in \mc{R}_{\red{x_0}}}(T_{\red{x_0}}(\xi))}\right)^{1/2}=A_{\red{x_0}},
\end{align*}
contradicting \eqref{e:ant}.
\end{proof}

Compactness of $M$, the fact that $\vol_{\Sigma_{\red{x_0}}}(\mc{R}_{\red{x_0}})\leq \vol(S^{n-1})$, and Corollary~\ref{l:volcano} imply Corollary~\ref{c:geomDep}.

\section{Dynamical and measure theoretic preliminaries}

\subsection{Dynamical preliminaries}

The following lemma gives an estimate on how much spreading the geodesic flow has near a point.

\begin{lem}
\label{l:flow}
Fix $\red{x_0}\in M$. Then there exists $\delta_{\red{M,p}}>0$ small enough and $C_1>0$ so that uniformly for $t\in [-\delta_{\red{M,p}},\delta_{\red{M,p}} ]$,
\begin{multline}
\label{e:flow-1} 
\frac{1}{2}d\big((\red{x_0},\xi_1),(\red{x_0},\xi_2)\big){-C_1}d\big((\red{x_0},\xi_1),(\red{x_0},\xi_2)\big)^2\leq d\big(G_t(\red{x_0},\xi_2),G_t(\red{x_0},\xi_1)\big)\\
\leq 2d\big((\red{x_0},\xi_1),(\red{x_0},\xi_2)\big){+C_1}d\big((\red{x_0},\xi_1),(\red{x_0},\xi_2)\big)^2
\end{multline}
{where $d$ is the distance induced by the Sasaki metric.}
Furthermore if $G_t(\red{x_0},\xi_i)=(x_i(t),\xi_i(t))$,
\begin{equation}
\label{e:flow-2} 
d_{{M}}(x_1(t),x_2(t)){\leq C_1}d\big((\red{x_0},\xi_1),(\red{x_0},\xi_2)\big)\delta_{\red{M,p}}
\end{equation}
{where $d_M$ is the distance induce by the metric $M$.}
\end{lem}

\begin{proof}
By Taylor's theorem
$$
G_t(\red{x_0},\xi_1)-G_t(\red{x_0},\xi_2)=d_\xi G_t(\red{x_0},\xi_2)(\xi_1-\xi_2)+O_{C^\infty}(\sup_{q\in \Sigma} |d_\xi ^2G_t(q)|(\xi_1-\xi_2)^2)
$$
Now,
$$
G_t(\red{x_0},\xi)=(\red{x_0},\xi)+(\partial_\xi p(\red{x_0},\xi) t,-\partial_xp(\red{x_0},\xi) t)+O(t^2)
$$
so
$$
d_\xi G_t(\red{x_0},\xi)=(0,I)+t(\partial^2_\xi p ,-\partial^2_{\xi x}p )+O(t^2)
$$
In particular,
$$
G_t(\red{x_0},\xi_1)-G_t(\red{x_0},\xi_2)=((0,I)+O(t))(\xi_1-\xi_2)+O((\xi_1-\xi_2)^2)
$$
and choosing $\delta_{\red{M,p}}>0$ small enough gives the result.
\end{proof}

\subsection{Measure theoretic preliminaries}

We will need a few measure theoretic lemmas to prove our main theorem.

\begin{lem}
\label{l:invariance}
Suppose that $\mu_{\red{x_0}}=\rho_{\red{x_0}}+fd\mc{H}^n_{\red{x_0}}$ is a finite Borel measure invariant under $G_t$ and $\rho_{\red{x_0}}\perp \mc{H}^n_{\red{x_0}}$. Then $\rho_{\red{x_0}}$ and $fd\mc{H}^n_{\red{x_0}}$ are invariant under $G_t$. 
\end{lem}

\begin{proof}
Since $\rho_{\red{x_0}}\perp \mc{H}^n_{\red{x_0}}$, there exist disjoint $N, P$ such that $\rho_{\red{x_0}}(P)=\mc{H}^n_{\red{x_0}}(N)=0$ and $\Lambda_{\red{x_0}}=N\cup P$. Suppose $A$ is Borel. Then the invariance of $\mu_{\red{x_0}}$ implies 
\begin{equation}
\label{e:inv}
\int(1_{A}\circ G_{-t}-1_{A})d\rho_{\red{x_0}} =\int (1_{A}-1_{A}\circ G_{-t})fd\mc{H}^n_{\red{x_0}}.
\end{equation}
Now, if $A\subset N$ then the fact that $G_t$ is a diffeomorphism implies {that it maps $0$ Hausdorff measure sets to $0$ Hausdorff measure sets and hence }$\mc{H}^n_{\red{x_0}}(A)=\mc{H}^n_{\red{x_0}}(G_t(A))=0$.
Therefore, 
\begin{equation}
\label{e:subN}\rho_{\red{x_0}}(A) =\rho_{\red{x_0}}(G_t(A)),\qquad A\subset N
\end{equation}
In particular,
 $$
 \rho_{\red{x_0}}(N)=\rho_{\red{x_0}}(G_t(N))=\rho_{\red{x_0}}(\Lambda_{\red{x_0}}).
 $$ 
Using again that for $t\in \R,$ $G_t:\Sigma\to \Sigma$ is a diffeomorphism, we have 
$$
\rho_{\red{x_0}}(G_t(P))=\rho_{\red{x_0}}(\Lambda_{\red{x_0}}\setminus G_t(N))=\rho_{\red{x_0}}(\Lambda_{\red{x_0}})-\rho_{\red{x_0}}(G_t(N))=0.
$$
So, in particular, 
\begin{equation}
\label{e:subP}
\rho_{\red{x_0}}(G_t(A))=0,\qquad A\subset P.\
\end{equation} 
Combining~\eqref{e:subN} with~\eqref{e:subP} proves that $\rho_{\red{x_0}}$ is $G_t$ invariant and hence~\eqref{e:inv} proves the lemma.
\end{proof}

Let $B(\xi,r)\subset \Sigma_{\red{x_0}}$ be the {ball of radius $r$ around $\xi$ for the distance induced by the Sasaki metric on $\Sigma_{\red{x_0}}$} and define 
\begin{equation}
\label{e:tubes}
T_{\red{\delta}}(\xi,r):=\bigcup_{t=-\red{\delta}}^{\red{\delta}} G_t(\{(\red{x_0},\xi_0)\mid \xi_0\in B(\xi,r)\}).
\end{equation}

\begin{lem}
\label{l:tubes}
Suppose $\red{\delta>0}$ and $ \rho_{\red{x_0}}$ is a finite measure invariant under $G_t$ and $\rho_{\red{x_0}}\perp \mc{H}^{n}_{\red{x_0}}$. Then for all $\e>0$, there exist $\xi_j\in \Sigma_{\red{x_0}}$ and $r_j>0$, $j=1,\dots$ so that 
\begin{gather}
\label{e:mucond}
\sum r_j^{n-1}<\e,\qquad \rho_{\red{x_0}}\left(\bigcup_j T_{\red{\delta}}(\xi_j,r_j)\right)=\rho_{\red{x_0}}(\Lambda_{\red{x_0},\red{\delta}}).
\end{gather}
\end{lem}

\begin{proof}
Fix $\delta>0$ so that 
$$
[-\delta,\delta]\times \Sigma_{\red{x_0}}\ni (t,\red{q})\mapsto G_t(\red{q})\in  \Lambda_{\red{x_0},\delta}
$$ 
is a diffeomorphism and use $(t,\red{q})$ as coordinates on $\Lambda_{\red{x_0},\delta}$. 

We integrate $\rho_{\red{x_0}}$ over $\Lambda_{\red{x_0},\delta}$ to obtain a measure on $\Sigma_{\red{x_0}}$. In particular, for $A\subset \Sigma_{\red{x_0}}$ Borel, define the measure
\begin{equation}
\label{e:defnu}
\tilde{\rho}_{\red{x_0}}(A):=\frac{1}{2\delta}\rho_{\red{x_0}}\left(\bigcup_{t=-\delta}^\delta G_t(A)\right).
\end{equation}
{Then, the invariance of $\rho_{\red{x_0}}$ implies that $\partial_t^*\rho_{\red{x_0}}=0$, \red{where for $F\in C_c^\infty(-\delta,\delta)\times \Sigma_{\red{x_0}}$, $\partial_t^*\rho_{\red{x_0}}(F)=\rho_{\red{x_0}}(\partial_tF).$} In particular, for all $F\in C_c^\infty (-\delta,\delta)\times \Sigma_{\red{x_0}}$, 
$$\int \partial_t Fd\rho_{\red{x_0}}=0.$$
Now, fix $\chi\in C_c^\infty(-\delta,\delta)$ with with $\int\chi dt=1$. Let $f\in C_c^\infty((-\delta,\delta)\times \Sigma_{\red{x_0}})$ and define 
$$\bar{f}(q):=\int f(t,q)dt.$$
Then $f(t,q)-\chi(t)\bar{f}(q)=\partial_t F$ with 
$$F(t,q):=\int_{-\infty}^t f(s,q)-\chi(s)\bar{f}(q)ds\in C_c^\infty((-\delta,\delta)\times \Sigma_{\red{x_0}}).$$
Therefore, for all $f\in C_c^\infty((-\delta,\delta)\times \Sigma_{\red{x_0}}) $ and $\chi \in C_c^\infty(-\delta,\delta)$ with $\int \chi dt=1$,
$$\int f(t,q)d\rho_{\red{x_0}}(t,q)=\int \chi(t)\bar{f}(q)d\rho_{\red{x_0}}(t,q) =\iiint f(s,q)ds\chi(t)d\rho_{\red{x_0}}(t,q).$$
Now, let $B\subset \Sigma_{\red{x_0}}$ be Borel, $I\subset (-\delta,\delta)$ Borel, and $f_n(t,q)\uparrow 1_{I}(t)1_{B}(q)$. Then by the dominated convergence theorem,
$$\rho_{\red{x_0}}(I\times B)=\iint |I|1_{B}(q)\chi(t)d\rho_x(t,q).$$
Next, let $\chi_n\uparrow \delta^{-1}1_{[0,\delta]}$ with $\int \chi_n \equiv 1$. Then we obtain
$$\rho_{\red{x_0}}(I\times B)=\frac{|I|}{\delta}\rho_{\red{x_0}}([0,\delta]\times B).$$
So, letting $\tilde{\rho}_{\red{x_0}}(B):=\delta^{-1}\mu([0,\delta]\times B)$, we have that for rectangles $I\times B$, $\rho_{\red{x_0}}(I\times B)=dt\times d\tilde{\rho}_{\red{x_0}}(I\times B)$. But then, since these sets generate the Borel sigma algebra,
}
\begin{equation}
\label{nuDesc}
\rho_{\red{x_0}}= dt\times \tilde{\rho}_{\red{x_0}}.
\end{equation}

Now, notice that $\mc{H}^n_{\red{x_0}}=g(t,\red{q})dt\times d\vol_{\Sigma_{\red{x_0}}}$ where $0<c<g\in C^\infty$. In particular, since 
$$
dt\times \tilde{\rho}_{\red{x_0}}\perp dt\times d\vol_{\Sigma_{\red{x_0}}}
$$
 we have that $\tilde{\rho}_{\red{x_0}}\perp d\vol_{\Sigma_{\red{x_0}}}$.

Thus, there exists $N,P\subset \Sigma_{\red{x_0}}$ so that $\tilde{\rho}_{\red{x_0}}(P)=\vol_{\Sigma_{\red{x_0}}}(N)=0$ and $\Sigma_{\red{x_0}}=N\sqcup P$. Hence for any $\e>0$, there exist $\xi_j\in \Sigma_{\red{x_0}}$ and $r_j>0$ so that 
$$
\sum_j r_j^{n-1}<\e, \quad \tilde{\rho}_{\red{x_0}}\left(\bigcup_jB(\xi_j,r_j)\right)=\tilde{\rho}_{\red{x_0}}(\Sigma_{\red{x_0}}).
$$
The lemma then follows from~\eqref{nuDesc} and invariance of $\rho_{\red{x_0}}$.
\end{proof}

\begin{lem}
\label{l:inv}
Suppose that $0\leq f\in L^1(\Lambda_{\red{x_0}},\mc{H}^n_{\red{x_0}})$ with $fd\mc{H}^n_{\red{x_0}}$ invariant under $G_t$. Then for $\delta_0>0$ small enough, write
$$
[-\delta_0,\delta_0]\times \Sigma_{\red{x_0}}\ni (t,q)\mapsto G_t(q)\in \Lambda_{\red{x_0}}
$$
for coordinates on $\Lambda_{\red{x_0},\delta_0}$. We have 
$$
f1_{\Lambda_{\red{x_0},\delta_0}}d\mc{H}^n_{\red{x_0}}=\tilde{f}(q)1_{[-\delta_0,\delta_0]}(t)dt\times d\vol_{\Sigma_{\red{x_0}}}
$$
where 
$$
\tilde{f}(q)=f(0,q)|\nu(H_p)|(0,q)
$$
and $\nu$ is a unit normal to $\Sigma_{\red{x_0}}\Subset \Lambda_{\red{x_0},\delta_0}$ with respect to the Sasaki metric. 
\end{lem}

\begin{proof}
Observe that $1_{\Lambda_{\red{x_0},\delta_0}}d\mc{H}^n_{\red{x_0}}$ is the volume measure on $\Lambda_{\red{x_0},\delta_0}$. Therefore,  $1_{\Lambda_{\red{x_0},\delta_0}}d\mc{H}^n_{\red{x_0}}\ll 1_{[-\delta_0,\delta_0]}(t)dt\times d\vol_{\Sigma_{\red{x_0}}}$ and in particular,
$$
f1_{\Lambda_{\red{x_0},\delta_0}}d\mc{H}^n_{\red{x_0}}=f(t,q)\frac{d\mc{H}^n_{\red{x_0}}}{dt\times d\vol_{\Sigma_{\red{x_0}}}}(t,q)1_{[-\delta_0,\delta_0]}(t)dt\times d\vol_{\Sigma_{\red{x_0}}}.
$$
Since $fd\mc{H}^n_{\red{x_0}}$ is invariant under $G_t${, it is invariant under translation in $t$ and we have}
$$
f(t,q)\frac{d\mc{H}^n_{\red{x_0}}}{dt\times d\vol_{\Sigma_{\red{x_0}}}}(t,q)=\tilde{f}(q)
$$
is constant in time.

To compute $\tilde{f}(q)$, we need only compute 
$$
\frac{d\mc{H}^n_{\red{x_0}}}{dt\times d\vol_{\Sigma_{\red{x_0}}}}(0,q).
$$
For this, observe that $1_{\Lambda_{\red{x_0},\delta}}\mc{H}^n_{\red{x_0}}$ is the volume measure on $\Lambda_{\red{x_0},\delta}$ with respect to the Sasaki metric. Therefore, we have $d\vol_{\Sigma_{\red{x_0}}}=N\lrcorner d\vol_{\Lambda_{\red{x_0},\delta_0}}$ where $N$ is a unit normal to $\Sigma_{\red{x_0}}$. More precisely, if $r\in C^\infty(\Lambda_{x_0,\delta_0})$ has $dr|_{\Sigma_{\red{x_0}}}(V)=\langle N,V\rangle_{g_s}$ where $g_s$ denotes the Sasaki metric and $V\in T_{\Sigma_{\red{x_0}}}\Lambda_{\red{x_0},\delta_0}$, then $\nu=dr|_{\Sigma_{\red{x_0}}}$ is a unit conormal to $\Sigma_{\red{x_0}}$ and
$$
\frac{d\mc{H}^n_{\red{x_0}}}{dt\times d\vol_{\Sigma_{\red{x_0}}}}(0,q)=|\partial_t (r\circ G_t)|_{t=0}|(q)=|\nu(H_p)|(q).
$$
\end{proof}

\section{A proof of Theorem~\ref{thm:local} for the Laplacian}
\label{s:laplace}

One can use a strategy similar to that in \cite{GT} to prove Theorem~\ref{thm:local} for eigenfunctions of the Laplacian. We sketch the proof in the case $\mu_{\red{x_0}}\perp \mc{H}^n_{\red{x_0}}$ for the convenience of the reader. {Following the arguments in Section~\ref{s:proof}, replacing Lemma~\ref{l:octopus} with~\eqref{QE3} it is possible to give a proof of the full theorem in this way. Note however, that much greater care would be needed to eliminate the dependence of the constant on $M$.  We wish to stress that the analysis in the next sections gives an effective geometric explanation for the gains in $L^\infty$ norms that is not available through use of the spectral projector. Moreover, it shows that the structure of the $L^\infty$ gains depends only on quantitative control on the transversality of the flow to the fibers.}

We start by constructing a convenient partition of unity. {This partition will also be used in the proof of the general case, so we write a careful proof.}

\begin{lem}
Fix $(x_0,\xi_j)\in  \Sigma_{x_0}$ and $r_j>0$, $j=1,\dots K<\infty$, $\delta>0$. Then there exist $\chi_j\in C_c^\infty(T^*M;[0,1]),$ $j=1\dots K$ so that 
\begin{equation}
\label{e:partition}
\begin{gathered} 
\supp \chi_j\cap \Lambda_{\red{x_0}}\subset T_{\red{4\delta}}(\xi_j,2r_j)\cap \Lambda_{{x_0},4\delta},\qquad H_p\chi_j\equiv 0 \text{ on }\Lambda_{x_0,3\delta}\\
\sum_j\chi_j\equiv 1\text{ on }\bigcup_{j=1}^KT_{\red{4\delta}}(\xi_j,r_j)\cap \Lambda_{x_0,3\delta},\qquad 0\leq \sum_j\chi_j\leq 1, \text{ on } \Lambda_{{x_0}}
\end{gathered}
\end{equation}
Furthermore, if 
\begin{equation}
\label{e:cover}
\bigcup_{j=1}^KT_{\red{4\delta}}(\xi_j,2r_j)\supset \Lambda_{x_0,3\delta},
\end{equation}
there exists $\chi_j$ {satisfying}~\eqref{e:partition} and
\begin{equation}
\label{e:partition2}
\sum_j\chi_j\equiv 1\text{ on } \Lambda_{x_0,3\delta}.
\end{equation}
\end{lem}

\begin{proof}
Let $\tilde{\chi}_j\in C_c^\infty(\Sigma_{x_0};[0,1])$ {satisfy} 
\begin{gather*}
\sum_j\tilde{\chi}_j\equiv 1 \text{ on }\bigcup_{j=1}^KB(\xi_j,r_j),\qquad \supp \tilde{\chi}_j\subset B(\xi_j,2r_j)\cap \Sigma_{x_0},\qquad 0\leq \sum_j\tilde{\chi}_j\leq 1.
\end{gather*}
Next, let $\psi\in C_c^\infty(\re;[0,1])$ with $\psi\equiv 1$ on $[-3\delta,3\delta]$ and $\supp \psi\subset (-4\delta,4\delta)$. For $\delta>0$ small enough, $G_t:[-4\delta,4\delta]\times \Sigma_{x_0}\to \Lambda_{x_0,4\delta}$ is a diffeomorphism and so we can define $\chi_j\in C_c^\infty(\Lambda_{x_0,4\delta};[0,1])$ by
$$
\chi_j(G_t({x_0},\xi))=\psi(t)\tilde{\chi}_j({x_0},\xi)
$$
so that $H_p\chi_j\equiv 0$ on $\Lambda_{x_0,3\delta}$. Finally, extend $\chi_j$ from $\Lambda_{x_0,4\delta}$ to a compactly supported function on $T^*M$ arbitrarily. Then $\chi_j$ $j=1,\dots K$ satisfy~\eqref{e:partition}.

If~\eqref{e:cover} holds, then we may take $\tilde{\chi}_j$ a partition of unity on $\Sigma_{x_0}$ subordinate to $B(\xi_j,2r_j)$ and hence obtain~\eqref{e:partition2} by the same construction.
\end{proof}

\begin{proof}[Sketch proof for Laplace eigenfunctions]
Fix $\delta >0$ and let $\rho \in S(\R)$ with $\rho(0)=1$ and $\supp\hat{\rho} \subset [\delta, 2 \delta].$
Let 
$$
S^*M(\gamma):= \{ (x,\xi); | |\xi|_{x} - 1| \leq \gamma \}
$$
 and  $\chi(x,\xi) \in C^{\infty}_{0}(T^*M)$ be a cutoff near the cosphere $S^*M$ with
$ \chi(x,\xi) =1$ for $(x,\xi) \in S^*M(\gamma)$ and $\chi(x,\xi) = 0$ when $(x,\xi) \in T^*M \setminus S^*M(2\gamma).$ 

Suppose that  $(-h^2\Delta_g-1)u_h=0$, and $u_h$ has defect measure $\mu$ with $\mu_{\red{x_0}}\perp \mc{H}^n_{\red{x_0}}.$ 
Then
\begin{align} 
\label{QE0}
u_h &= \rho( h^{-1} [ {\sqrt{-h^2 \Delta_g}} -1]) u_h = \int_{\R} \hat{\rho}(t)  e^{it [{\sqrt{-h^2\Delta_g}}-1]/h} \chi(x,hD) u_h \, dt + O_{\gamma}(h^{\infty}). 
\end{align}

Setting $V(t,x,y,h):=  \Big( \hat{\rho}(t)  e^{it [{\sqrt{-h^2\Delta_g}}-1]/h} \chi(x,hD) \Big) (t,x,y),$  by propagation of singularities,
$$
WF_h'( V(t,\cdot, \cdot,h)) \subset \{ (x,\xi,y,\eta); (x,\xi) = G_t(y,\eta),  \,  | |\xi|_{{x}} -1 | \leq 2\gamma \, , t \in [\delta, 2\delta] \}.
$$
Let $b_{{\red{x_0}},\gamma}(y,\eta) \in C_c^\infty(T^*M)$ have
$$
\supp b_{\red{x_0},\gamma} \subset \{(y,\eta)\mid (y,\eta)=G_t(x,\xi)\, \text{for some} \, (x,\xi)\in S^*M(3\gamma)\,\text{with}\, {d_M}(x,x_0)<{2}\gamma, |t|\leq {4}\delta\}
$$ 
with 
$$
b_{\red{x_0},\gamma}\equiv 1\text{ on }\{(y,\eta)\mid (y,\eta)=G_t(x,\xi)\, \text{for some} \, (x,\xi)\in S^*M(2\gamma)\,\text{with}\, {d_M}(x,x_0)<\gamma, |t|\leq {3}\delta\}.
$$

Then, by wavefront calculus, it follows that
\begin{equation} 
\label{microlocalize-t}
u_h(\red{x_0}) = \int_{M}  \bar{V}(\red{x_0},y,h)  \, b_{\red{x_0},\gamma}(y,hD_y) u_h(y) dy  + O_{\gamma}(h^{\infty}), 
\end{equation}
where,
$$
\bar{V}(x,y,h) := \int_{\R} \hat{\rho}(t) \big( e^{it [{\sqrt{-h^2\Delta_g}}-1]/h} \chi(x,hD) \big) (t,x,y) \, dt.
$$

By a standard stationary phase argument \cite[Chapter 5]{SoggeBook},
\begin{equation} 
\label{wkb-t}
\bar{V}(x,y,h) = h^{\frac{1-n}{2}} \sum_{\pm}e^{\pm i \red{d_M}(x,y)/h}   a_{\pm}(x,y,h) \,  \hat{\rho}(\red{d_M}(x,y)) + O_{\gamma}(h^{\infty}), 
\end{equation}
where $ a_{\pm}(x,y,h) \in S^{0}(1)$.

Then, in view of~\eqref{wkb-t} and~\eqref{microlocalize-t},
\begin{equation} 
\label{QE1}
u_h(\red{x_0}) = (2\pi h)^{\frac{1-n}{2}} \sum_{\pm}\int_{ \delta < |y-\red{x_0}| < 2\delta} e^{\pm i\red{d_M}(\red{x_0},y)/h}  a_{\pm}(\red{x_0},y,h) \hat{\rho}(\red{d_M}(\red{x_0},y)) \, b_{x,\gamma}(y,hD_y) u_h(y) dy  + O_{\gamma}(h^{\infty}). 
\end{equation}

Let $\chi_j$, be as in~\eqref{e:partition} with $T_{\red{4\delta}}(\xi_j,r_j)$ satisfying~\eqref{e:foot} and $\sum r_j^{n-1}<\e$. Define $\psi=1-\sum_j\chi_j$. Then
$$
u_h(\red{x_0}) = \sum_{\pm} I_{\pm} +II_{\pm}+O_{\gamma}(h^\infty)
$$
where
\begin{equation}
 \label{QE1}
\begin{aligned}
I_{\pm}&=(2\pi h)^{\frac{1-n}{2}} \int_{ \delta < |y-\red{x_0}| < 2\delta} e^{\pm i\red{d_M}(\red{x_0},y)/h} a_{\pm}(\red{x_0},y,h) \hat{\rho}(\red{d_M}(\red{x_0},y)) \,\psi(y,hD_{y})b_{\red{x_0},\gamma}(y,hD_y) u_h(y) dy \\
II_{\pm}&=\sum_j(2\pi h)^{\frac{1-n}{2}} \int_{ \delta < |y-\red{x_0}| < 2\delta} e^{\pm i\red{d_M}(\red{x_0},y)/h} a_{\pm}(\red{x_0},y,h) \hat{\rho}(\red{d_M}(\red{x_0},y)) \,\chi_j(y,hD_{y})b_{\red{x_0},\gamma}(y,hD_y) u_h(y) dy{.}
\end{aligned} 
\end{equation}
An application of  Cauchy-Schwarz to $I_{\pm}$ gives
\begin{align} 
\label{QE2}
\limsup_{h\to 0}h^{\frac{n-1}{2}}|I_{\pm}| &\leq  \, C   \limsup_{h\to 0}\| \psi(y,hD_{y})b_{\red{x_0},\gamma}(y,hD_y) u_h \|_{L^2}{.}
\end{align} 
{Next observe that}
\begin{align*}
\limsup_{\gamma\to 0}\limsup_{h\to 0}\|\psi(y,hD_{y})b_{\red{x_0},\gamma}(y,hD_y)u\|_{L^2}^2&=\limsup_{\gamma\to 0}\int_{S^*M} |\psi|^2|b_{\red{x_0},\gamma}(y,\xi)|^2d\mu\\
&\leq C{\mu}( \supp \psi\cap \Lambda_{\red{x_0},4\delta})\\
&\leq C {\mu_{\red{x_0}}}(\supp \psi) \leq C\e{.}
\end{align*}
{Note that the first inequality follows from the fact that $\lim_{\gamma\to 0}b_{\red{x_0},\gamma}\leq 1_{\Lambda_{\red{x_0},4\delta}}$.}
On the other hand, by propagation of singularities, for each $\chi_j$ in $II_{\pm}$, we may insert $\varphi_j\in C_c^\infty(M)$ localized to 
$$\pi (T_{\red{4\delta}}(\xi_j,r_j)\cap \{\delta<{d_M}(x,x_0)<2\delta\}){,}$$
 where $\pi:T^*M\to M$ is projection to the base. In particular, replacing $\chi_j(y,hD_y)$ by $\varphi_j(y)\chi_j(y,hD_y)$ and applying Cauchy-Schwarz to each term of $II$, we have
\begin{align} 
\label{QE3}
\limsup_{h\to 0}h^{\frac{n-1}{2}}|II_{\pm}| &\leq  C\sum_j\|\varphi_j\|_{L^2}  \limsup_{h\to 0}\| \chi_j b_{\red{x_0},\gamma}(y,hD_y) u_h \|_{L^2}{.}
\end{align} 
Now, since $\varphi_j$ is supported on a tube of radius $r_j$, $\|\varphi_j\|_{L^2}\leq C r_j^{(n-1)/2}$. Furthermore,
\begin{align*}
\lim_{\gamma\to 0}\lim_{h\to 0}\|\chi_j(y,hD_{y})b_{\red{x_0},\gamma}(y,hD_y)u\|_{L^2}^2&=\lim_{\gamma\to 0}\int_{S^*M} \chi_j^2|b_{\red{x_0},\gamma}(y,\xi)|^2d\mu\leq \int_{\Lambda_{\red{x_0}}} \chi_j^2d\mu{.}
\end{align*}
Thus, applying Cauchy-Schwarz once again to the sum in~\eqref{QE3},
$$
\limsup_{h\to 0}h^{\frac{n-1}{2}}|II_{\pm}| \leq C\left(\sum_j r_j^{n-1}\right)^{1/2}\left(\int \sum_j \chi_j^2d\mu\right)^{1/2}\leq C\e^{1/2}.
$$
Sending $\e\to 0$ proves the theorem.
\end{proof}

\section{$L^\infty$ estimates microlocalized to $\Lambda_{\red{x_0}}$}

For the next two sections, we assume that $u$ is compactly microlocalized and $Pu=o_{L^2}(h)$ where $P$ is as in Theorem~\ref{thm:local}.

\begin{lem} \label{vanish}
Suppose that $P$ is as in Theorem~\ref{thm:local}, $u$ is compactly microlocalized, and $Pu=o_{L^2}(h)$. Then for $q,a\in S^\infty(T^*M)$ 
\begin{align*}
\|a(x,hD)q(x,hD)u\|_{L^2}^2&=\int |a|^2|q|^2d\mu +o(1),\\
\|a(x,hD)Pq(x,hD)u\|_{L^2}^2&=h^2\int |a|^2|H_pq|^2d\mu +o(h^2).
\end{align*}
\end{lem}

\begin{proof}
First observe that since $u$ is compactly microlocalized, there exists $\chi \in C_c^\infty(T^*M)$ so that 
$$u=\chi(x,hD)u+O_{\mc{S}}(h^\infty).$$
Therefore, we may assume $q,a\in C_c^\infty(T^*M)$. 
The first equality then follows from the definition of the defect measure and the fact that $[a(x,hD)]^*=\bar{a}(x,hD)+O_{L^2\to L^2}(h)$.
For the second, note that
\begin{align*}
Pq(x,hD)u&=q(x,hD)Pu+[P,q(x,hD)]u\\
&=q(x,hD)Pu+\frac{h}{i}\{p,q\}(x,hD)u+O_{L^2}(h^2).
\end{align*}
The lemma follows since $Pu=o_{L^2}(h)$.
\end{proof}

 At this point, following the argument in Koch--Tataru--Zworski \cite{KTZ}, we work $h$-microlocally. The first step is to reduce the $L^2 \to L^\infty$ bounds to a neighbourhood of $\Sigma=\{p=0\}.$
 
\begin{lem}
Suppose that $u$ is compactly microlocalized and $Pu=o_{L^2}(h).$
Then for $\chi_{\Sigma}\in C_c^\infty(T^*M)$ with $\chi_{\Sigma}\equiv 1$ in a neighborhood of $\Sigma=\{p=0\}$,
\begin{equation} 
\label{l2linfty}
\| (1-\chi_{\Sigma}(x,hD) ) u \|_{L^\infty}=o(h^{\frac{2-n}{2}}).
\end{equation}
\end{lem}

\begin{proof}
Since $u$ is compactly microlocalized, there exists $\chi \in C_c^\infty(T^*M)$ so that 
$$
u=\chi(x,hD)u+O_{\mc{S}}(h^\infty{\|u\|_{L^2(M)}}).
$$
For $\chi_{\Sigma}\in C_c^\infty(T^*M)$ with $\chi_{\Sigma}\equiv 1$ in a neighborhood of $\Sigma$, $|p|\geq c>0$ on $\supp (1-\chi_{\Sigma})\chi.$
Therefore, by the elliptic parametrix construction, for any $q\in S^\infty(T^*M)$, there exists $e\in C_c^\infty(T^*M)$ so that 
$$
e(x,hD) P =(1-\chi_{\Sigma})(x,hD)q(x,hD)\chi(x,hD) +O_{\mc{D}'\to \mc{S}}(h^\infty)
$$ 
and in particular,  
\begin{equation}
\label{e:elliptic}
(1-\chi_{\Sigma})(x,hD) q(x,hD)u=o_{L^2}(h).
\end{equation}

The compact microlocalization of $u$ together with~\eqref{e:elliptic} {(for $q\equiv 1$)} and the Sobolev estimate \cite[Lemma 7.10]{EZB} implies
\begin{equation*} 
\| (1-\chi_{\Sigma}(x,hD) ) u \|_{L^\infty}{\leq Ch^{-\frac{n}{2}}\|(1-\chi_{\Sigma}(x,hD))u\|_{L^2(M)}}=o(h^{\frac{2-n}{2}}).
\end{equation*}
\end{proof}

To simplify the writing somewhat, we introduce the notation $u_{\Sigma}:= \chi_{\Sigma}(x,hD) u.$

\subsection{Microlocal $L^\infty$ bounds near $\Sigma$}

In view of (\ref{l2linfty}), it suffices to consider points in an arbitrarily small tubular neighborhood of $\Sigma = \{p=0 \}.$ More precisely, we cover $\supp \chi_{\Sigma}$ by a union $\cup_{j=0}^{N} B_j$ of open balls $B_j$ centered at points $(x_j,\xi_j) \in \Sigma \subset \{p=0 \}.$ We let $\chi_j \in C^{\infty}_0(B_j)$ be a corresponding partition of unity with 
\begin{equation}
\nonumber
u_{\Sigma} = \sum_{j=0}^{N} \chi_j(x,hD) u_{\Sigma}+O_{\mc{S}}(h^\infty)
\end{equation}
By possible refinement, the supports of $\chi_j$ can be chosen arbitrarily small.

Since the argument here is entirely local, it suffices to $h$-microlocalize to supp $\chi_0 \subset B_0$ where $B_0$ has  center
$(x_0,\xi_0)\in \{p=0\}$. Since we have assumed $\partial_\xi p\neq 0$ in $\{p=0\}$, we may assume that $\partial_{\xi_1}p(x_0,\xi_0)\neq 0$ and $\partial_{\xi'}p(x_0,\xi_0)=0$. Therefore, choosing $\supp \chi $ supported sufficiently close to $(x_0,\xi_0)$, it follows from the implicit function theorem that
$$
p\chi =e(x,\xi)(\xi_1-a(x,\xi'))
$$
with $ e(x,\xi) $ elliptic on $ \supp \chi_0$ provided the latter support is chosen small enough. Thus,
$$
P\chi_0=E(x,hD)(hD_{x_1}-a(x,hD_{x'}))\chi_0(x,hD)+hR\chi_0(x,hD).
$$
{Note that by adjusting $R$, we may assume that for each fixed $x_1$, $a(x_1, y,hD_{y'})$  is self adjoint on $L^2_{y'}$.}
Therefore, 
$$
(hD_{x_1}-a(x_1,x',hD_{x'}))\chi_0q(x,hD)u=E^{-1}(x,hD)P\chi_0 q(x,hD)u+hR_1\chi_0(x,hD)q(x,hD)u.
$$

In particular, from the standard energy estimate (see for example \cite[Lemma 3.1]{KTZ}) with $(x_1,x') \in \R^n,$
\begin{multline} 
\label{energy}
\| \chi_0 q(x,hD) u_{\Sigma} (x_1=s, \cdot) \|_{L^2_{x'}(\R^{n-1})} \leq  \|  \chi_0 q(x,hD) u_{\Sigma} (x_1=t,\cdot)  \|_{L^2_{x'}(\R^{n-1})} \\ 
 +  Ch^{-1} |s-t|^{1/2} (\| P \chi_0 q(x,hD) u_{\Sigma} \|_{L^2_x(\R^{n})}+h\|R_1\chi_0q(x,hD)u_{\Sigma}\|_{L^2_x(\R^n)}). 
\end{multline} 

\subsection{Microlocalization to the flowout}

Our next goal will be to insert microlocal cutoffs restricting to a neighborhood of $\Lambda_{x_0,\delta}$ for some $\delta>0$ into the right hand side of~\eqref{energy}.  

Let $\e\ll \delta$, $\chi_{\e,x_0}\in C_c^\infty(M;[0,1]))$ with 
$$
\chi_{\e,x_0}\equiv 1\text{ on }B(x_0,\e),\qquad \supp \chi_{\e,x_0}\subset B(x_0,2\e).
$$ 
Let $b_{\e,x_0}\in C_c^\infty(T^*M;[0,1])$ with 
\begin{gather}
\supp b_{\e,x_0}\cap \{p=0\}\subset\bigcup_{x\in B(x_0,3\e)}\Lambda_{x,3\delta},\qquad\supp b_{\e,x_0}\subset {\big\{q\mid d(q,|p|\leq \e)<2\e\big\}},\label{e:platypus1}\\
 b_{\e,x_0}\equiv 1\text{ on }{\Big\{q\,\big|\, d\Big(q,\bigcup_{t=-2\delta}^{2\delta}G_t\left\{(x,\xi)\mid  |p(x,\xi)|\leq\e, d(x,x_0)<2\e\right\}\Big)<\e\Big)\Big\}}\label{e:platypus2}.
\end{gather}

\begin{lem}
\label{l:squirrel}
There exists ${C_0}>0,\e_0>0$ and {$U$ a neighborhood of $(x_0,\xi_0)$} so that for all $\chi_0\in C_c^\infty(T^*M)$ supported in {$U$}, $0<\e\ll \delta<\e_0$, $\chi_{\e,x_0}$, $b_{\e,x_0}$ as above, $q\in S^\infty(T^*M)$, and $y_1\in \R$
\begin{multline} 
\label{UPSHOT 1}
\|(q\chi_{\e,x_0}\chi_0)(x,hD)u_{\Sigma}|_{x_1=y_1} \|_{L^2_{x'}(\R^{n-1})}\leq 2\delta_0^{-1/2}\|b_{\e,x_0}(x,hD)q(x,hD)\chi_0(x,hD)u_{\Sigma}\|_{L^2_x(\R^n)}\\
+{C_0}\delta_0^{\frac{1}{2}}h^{-1}\|b_{\e,x_0}(x,hD)Pq(x,hD)\chi_0(x,hD)u_{\Sigma}\|_{L^2_x(\R^n)}+o_{\e,\delta}(1)
\end{multline}
where $\delta_0:=\delta|\partial_{\xi}p(x_0,\xi_0)|_g$ and $|\partial_{\xi}p|_g:=| \partial_{\xi}p\cdot \partial_x|_g.$
\end{lem}

\begin{rem}
\begin{itemize}
\item[-]  In (\ref{UPSHOT 1}), the local defining functions $x_1$ depend on $j$, but we will abuse notation somewhat and suppress the dependence on the index.
\item[-] {Note that the constant $C_0$ may depend on $P$ and $M$ in unspecified ways. In order to remove this dependence in Theorem~\ref{thm:local}, we choose $\delta$ sufficiently small when applying Lemma~\ref{l:squirrel}.}
\end{itemize}
\end{rem}

\begin{proof}
Let 
$$
A(x_1,y_1,x',hD_{x'}):=-\int_{y_1}^{x_1}a(s,x',hD_{x'})ds
$$ 
and $w=\chi_0q(x,hD)u_{\Sigma}$. Then 
$$
w(y_1,x')=e^{-\frac{i}{h}A(t,y_1,x',hD_{x'})}w|_{x_1=t}\\
-\frac{i}{h}\int_{y_1}^{t}e^{-\frac{i}{h}A(s,y_1,x',hD_{x'})}f(s,x')ds
$$
where 
\begin{equation}
\label{e:f}
f(x):=E^{-1}(x,hD)P\chi_0 q(x,hD)u_{\Sigma} +hR_1\chi_0(x,hD)q(x,hD)u_{\Sigma}.
\end{equation}
{Moreover, since we have arranged that $a(s,x',hD_{x'})$ is self adjoint for each fixed $s$, $e^{-\frac{i}{h}A(s,y_1,x',hD_{x'})}$ is unitary.}

Let $\delta_0:=\delta |\partial_{\xi}p(x_0,\xi_0)|_g$ and $\psi\in C_c^\infty(\re;[0,1])$ with $\supp \psi \subset [0,\delta_0 ]$ and $\int \psi=1$. Then, integrating in ${t}$,
\begin{equation}
\label{e:squid2}
w(y_1,x')=\int \psi(t)e^{-\frac{i}{h}A(t,y_1,x',hD_{x'})}w|_{x_1=t}dt-\frac{i}{ h}\int \psi(t) \int_{y_1}^{t}e^{-\frac{i}{h}A(s,y_1,x',hD_{x'})}f(s,x')dsdt
\end{equation}

Now, let $\tilde{b}_{\e,x_0}$ satisfy
\red{\begin{equation}
\label{e:platypus3}
 \tilde{b}_{\e,x_0}\equiv 1\text{ on }{\Big\{q\,\big|\, d\Big(q,\bigcup_{t=-2\delta}^{2\delta}G_t\left\{(x,\xi)\mid  |p(x,\xi)|\leq\e, d(x,x_0)<2\e\right\}\Big)<\e/2\Big)\Big\}}.
 \end{equation}}
and have $\supp \tilde{b}_{\e,x_0}\subset \{b_{\e,x_0}\equiv 1\}.$ \red{This is possible by ~\eqref{e:platypus2}.}
{We next aim to prove}
\begin{equation}
\label{e:est0}
\begin{aligned} 
\chi_{\e,x_0}w(y_1,x')&=\int \psi(t) \chi_{\e,x_0} e^{-\frac{i}{h}A(t,y_1,x',hD_{x'})}({\tilde{b}}_{\e,x_0}(x,hD)w)|_{x_1=t}dt\\
&\qquad -\frac{i}{h}\chi_{\e,x_0} \int \psi(t)\int_{y_1}^{t}e^{-\frac{i}{h}A(s,y_1,x',hD_{x'})}({\tilde{b}}_{\e,x_0}(x,hD)f)(s,x')dsdt+o_{\e,\delta}(1)_{L^\infty_{y_1}L^2_{x'}}
\end{aligned}
\end{equation}
{To do this, we show that} for $q_1\in S^0(T^*M)$, $s\in[0,\delta_0]$
\begin{equation}
\label{e:squid}
\chi_{\e,x_0}(y_1,x')e^{-\frac{i}{h}A(s,y_1,x',hD_{x'})}({I}-{\tilde{b}}_{\e,x_0}(x,hD))\chi_0(x,hD)q_1(x,hD){u_{\Sigma}}=o_\e(h)_{L^2_x}.
\end{equation}
Let $\varphi\in C_c^\infty(\R)$ with $\varphi\equiv 1$ on $[-1,1]$. By~\eqref{e:elliptic} 
$$
\chi_{\e,x_0}(y_1,x')e^{-\frac{i}{h}A(x_1,x',hD_{x'})}({I}-{\tilde{b}}_{\e,x_0}(x,hD))\chi_0(x,hD)q_1(x,hD)({I}-\varphi(\e^{-2}p(x,hD))){u_{\Sigma}}=o_{\e}(h)_{L_x^2}.
$$
Therefore, we need only estimate
\begin{equation}
\label{e:elliptPart}
\chi_{\e,x_0}(y_1,x')e^{-\frac{i}{h}A(s,y_1,x',hD_{x'})}({I}-{\tilde{b}}_{\e,x_0}(x,hD))\chi_0(x,hD)q_1(x,hD)\varphi(\e^{-2}p(x,hD)){u_{\Sigma}}.
\end{equation}

{In order to estimate~\eqref{e:elliptPart}, we apply propagation of singularities for $e^{-\frac{i}{h}A}.$}
Let $\tilde{G}_t$ denote the Hamiltonian flow of $\xi_1-a(x,\xi')$. {We show that for $\delta$ small enough and $|t|\leq \delta_0$, 
\begin{equation}
\label{e:modifiedDyn}
\supp \chi_{\e,x_0}\cap \tilde{G}_t(\supp(1-\tilde{b}_{\e,x_0})\varphi(\e^{-2}p)\chi_0)=\emptyset.
\end{equation}
Since $\supp \psi \subset [0,\delta_0]$ propagation of singularities then implies that 
\begin{equation}
\label{e:propPart}
\psi(s)\chi_{\e,x_0}(y_1,x')e^{-\frac{i}{h}A(s,y_1,x',hD_{x'})}({I}-{\tilde{b}}_{\e,x_0}(x,hD)){q_1}(x,hD)\varphi(\e^{-2}p(x,hD)){u_{\Sigma}}=O_{\e}(h^\infty)_{L^2_x}.
\end{equation}
}

{We now prove~\eqref{e:modifiedDyn}. For $p(y_0,\eta_0)=0$, if $G_t(y_0,\eta_0)=(x_1(t),x'(t),\xi(t))$, then $\tilde{G}_{x_1(t)}(y_0,\eta_0)=(x_1(t),x'(t),\xi(t)).$ Since we assume that $\partial_{\xi'}p(x_0,\xi_0)=0$, $\partial_{\xi_1}p(x_0,\xi_0)\neq 0$ {we may choose $U$ small enough so that for $q\in U$
$$
\frac{2}{3}|\partial_{\xi}p(x_0,\xi_0)|_g\leq |\partial_{\xi_1}p(q)|\leq \frac{3}{2}|\partial_{\xi}p(x_0,\xi_0)|_g.
$$
Thus, for $q\in \supp\chi_0$, 
$$
\frac{2}{3}|\partial_{\xi}p(x_0,\xi_0)|_gt+O(t^2)\leq |x_1(G_t(q))-x_1(q)|\leq \frac{3}{2}|\partial_{\xi}p(x_0,\xi_0)|_gt+O(t^2)
$$ }
Now, suppose $q\in \{|p|\leq C\e^2\}\cap\supp \chi_0$ so that $\tilde{G}_t(q)\in \supp\chi_{\e,x_0}$ for some $|t|\leq \delta_0$. Then, there exists $t\in [-2\delta,2\delta]$ and $C>0$ such that
$$
d (x(G_t(q)), x_0)\leq \e +C\e^2.
$$ 
In particular, by~\eqref{e:platypus3}, ${\tilde{b}}_{\e,x_0}\equiv 1$ in a neighborhood of $q$ and hence $q\notin \supp (1-\tilde{b}_{\e,x_0}).$ In particular, this proves~\eqref{e:modifiedDyn} and hence~\eqref{e:propPart}.
}

Together~\eqref{e:elliptPart} and~\eqref{e:propPart} {give~\eqref{e:squid} and in particular, applying~\eqref{e:squid} with $q_1=1$ for the first term in~\eqref{e:squid2} and 
$$
q_1=E^{-1}(x,hD)P\chi_0 q(x,hD) +hR_1\chi_0(x,hD)q(x,hD).
$$
for the second term in~\eqref{e:squid2} gives}
 ~\eqref{e:est0}. {In turn,~\eqref{e:est0}} implies
\begin{equation*}
\|\chi_{\e,x_0} w(y_1,\cdot)\|_{L^2_{x'}(\R^{n-1})}\leq \delta_0^{-1/2}\|{\tilde{b}}_{\e,x_0}(x,hD)w\|_{L^2_x(\R^n)}+{C_0}\delta_0^{\frac{1}{2}}h^{-1}\|{\tilde{b}}_{\e,x_0}(x,hD)f\|_{L^2_x(\R^n)}+o_{\e,\delta}(1).
\end{equation*}

Now, 
$$
q(x,hD)\chi_{\e,x_0}\chi_0(x,hD)u_{\Sigma}=\chi_{\e,x_0}\chi_0(x,hD)q(x,hD)u_\Sigma+[q(x,hD),\chi_{\e,x_0}\chi_0(x,hD)]u_\Sigma.
$$
Therefore, {applying the Sobolev embedding~\cite[Lemma 7.10]{EZB} in 1 dimension}
$$
\|[q(x,hD),\chi_{\e,x_0} \chi_0(x,hD)]u_{\Sigma} (x_1, \cdot )\|_{L_{x'}^2(\R^{n-1})} = O_\e(h^{1/2}),
$$
we have the following $L^2$ bound along the section $x_1=y_1$ of $\supp \chi_0 \subset \supp\chi_{\Sigma}.$
\begin{multline} 
\label{local upshot 1}
\|q(x,hD)\chi_{\e,x_0}\chi_0(x,hD)u_{\Sigma}(y_1, \cdot) \|_{L^2_{x'}(\R^{n-1})}\leq \\\delta_0^{-1/2}\|{\tilde{b}}_{\e,x_0}(x,hD)w\|_{L^2_x(\R^n)}+{C_1}\delta_0^{1/2}h^{-1}\|{\tilde{b}}_{\e,x_0}(x,hD)f\|_{L^2_x(\R^n)}
 +o_{\e,\delta}(1). 
 \end{multline}
 
 {Observe that, 
 \begin{align*}
&\| {\tilde{b}}_{\e,x_0}(x,hD)f\|&\\
&\qquad\leq \|{\tilde{b}}_{\e,x_0}(x,hD)E^{-1}(x,hD)P\chi_0 q(x,hD)u_{\Sigma}\| +h\|{\tilde{b}}_{\e,x_0}(x,hD)R_1\chi_0(x,hD)q(x,hD)u_{\Sigma}\|\\
&\qquad\leq C_2\|b_{\e,x_0}(x,hD)P\chi_0(x,hD)q(x,hD)u_{\Sigma}\|+C_2h\|b_{\e,x_0}(x,hD)\chi_0(x,hD)q(x,hD)u_{\Sigma}\|+o_{\e,\delta}(h)
\end{align*}
and
$$
\|{\tilde{b}}_{\e,x_0}(x,hD)w\|_{L^2_x(\R^n)}\leq \|b_{\e,x_0}(x,hD)\chi_0(x,hD)q(x,hD)u_{\Sigma}\|+o_{\e,\delta}(1).
$$
Taking $\e_0>0$ so small so that for $\delta<\e_0$, $C_1C_2\delta_0^{1/2}\leq \delta_0^{-1/2}$, and letting $C_0=C_1C_2$ we have 
\begin{multline} 
\label{local upshot 2}
\|q(x,hD)\chi_{\e,x_0}\chi_0(x,hD)u_{\Sigma}(y_1, \cdot) \|_{L^2_{x'}(\R^{n-1})}\leq 2\delta_0^{-1/2}\|b_{\e,x_0}(x,hD)\chi_0(x,hD)q(x,hD)u_{\Sigma}\|_{L^2_x(\R^n)}\\+C_0\delta_0^{1/2}h^{-1}\|b_{\e,x_0}(x,hD)P\chi_0(x,hD)q(x,hD)u_{\Sigma}\|_{L^2_x(\R^n)}
 +o_{\e,\delta}(1). 
 \end{multline}
}
\end{proof}

\begin{lem}
\label{l:flowoutNeed}
Suppose that for some $\delta>0$, $q\in S^0(T^*M)$ has $q \equiv 0$ on $\Lambda_{x_0,3\delta}$. Then for $r(h)=o(1)$. 
$$
\limsup_{h\to 0} h^{\frac{n-1}{2}}\|q(x,hD)u_\Sigma\|_{L^\infty(B(x_0,r(h)))}=0.
$$
\end{lem}

\begin{proof}
Observe that Lemma~\ref{l:squirrel} gives for each $j=1,\dots N$, 
\begin{multline*}
\|(q\chi_{\e,x_0}\chi_j)(x,hD)u_{\Sigma}|_{x_1=y_1} \|_{L^2_{x'}(\R^{n-1})}\leq 2\delta_0^{-1/2}\|b_{\e,x_0}(x,hD)q(x,hD)\chi_j(x,hD)u_{\Sigma}\|_{L^2_x(\R^n)}\\
+{C_0}\delta_0^{\frac{1}{2}}h^{-1}\|b_{\e,x_0}(x,hD)Pq(x,hD)\chi_j(x,hD)u_{\Sigma}\|_{L^2_x(\R^n)}+o_{\e,\delta}(1).
\end{multline*}
{Observe that since $r(h)=o(1)$, for $h$ small enough, $\chi_{\e,x_0}\equiv 1$ on $B(x_0,r(h))$. Hence,} applying the Sobolev estimate \cite[Lemma 7.10]{EZB} and Lemma~\ref{vanish} gives
\begin{multline*}
\limsup_{h\to 0}h^{n-1}\|(q\chi_j)(x,hD)u_{\Sigma} \|^2_{L^\infty(B(x_0,r(h)))}\leq 2\delta_0^{-1}\int b^2_{\e,x_0}(x,hD)q^2(x,hD)\chi^2_jd\mu\\
+{C_0}\delta_0\int b^2_{\e,x_0}(x,hD)|H_p (q(x,hD)\chi_j)|^2d\mu.
\end{multline*}
Sending $\e\to 0$ and using the dominated convergence theorem proves the lemma since $\mu(T^*M)=1<\infty$, $\lim_{\e\to 0}b^2_{\e,x_0}\leq 1_{\Lambda_{x_0,3\delta}}$, {$H_p$ is tangent to $\Lambda_{x_0}$}, and {so the fact that} $q$ vanishes identically on $\Lambda_{\red{x_0},3\delta}$ {implies the same for $H_pq$}.
\end{proof}

\section{Decomposition into wave packets}

We now choose a convenient partition $\chi_j$ and functions $q_{j,i}$, $i=2,\dots n$ to prove the main theorem. The $\chi_j$ localize to individual bicharacteristics, and $\sum_i q_{j,i}$ will measure concentration in neighborhoods of each bicharacteristic. We then show that understanding the mass localization to finer and finer neighborhoods of geodesics yields the structure of the defect measure. 

\subsection{$L^\infty$ contributions near {bicharacteristics}}

We need the following version of the $L^\infty$ Sobolev embedding.

\begin{lem}
\label{l:improvedLinfty}
{There exists $C_{n,l}>0$ depending only on $n$ and $l$ so that for all} $v\in H^l(\re^{n-1})$ with $l>(n-1)/2$ and all $\e>0$ 
$$
\|v\|_{L^\infty}^2\leq C_{n,l}h^{-n+1}\Big(\e^{n-1}\|v\|_{L^2}^2+\e^{n-2l-1}\sum_{i=1}^{n-1}\|(hD_{x_i})^l v\|_{L^2}^2\Big).
$$
In particular this holds if $v$ is compactly microlocalized.
\end{lem}

\begin{proof}
Let $\zeta\in C_c^\infty([-2,2])$ with $\zeta\equiv 1$ on $[-1,1]$ and $\zeta_\e(x)=\zeta(\e^{-1}x).$ 

Then 
\begin{align*} 
v(x)&=(2\pi h)^{-n+1}\int e^{i\langle x,\xi\rangle/h}[\zeta_\e(|\xi|)+(1-\zeta_\e(|\xi|))]\mc{F}_h(v)(\xi)d\xi
\end{align*}
Applying the triangle inequality and Cauchy--Schwarz, and letting $w_l(\xi)=\sqrt{\sum_{i=1}^{n-1}\xi_i^{2l}}$
\begin{align}
\|v\|_{L^\infty}^2&\leq h^{-2(n-1)}(\e^{n-1}\|\zeta\|_{L^2}^2\|\mc{F}_h v\|_{L^2}^2+\|(1-\zeta_\e)w_l^{-1}\|_{L^2}^2\|w_l\mc{F}_hv\|_{L^2}^2)\label{e:sobolevEst}
\end{align}
Now, 
\begin{gather*}
\|(1-\zeta_\e)w_l^{-1}\|_{L^2}^2=\e^{n-2l-1}\|(1-\zeta)w_l^{-1}\|_{L^2}^2\\
\|w_l\mc{F}_hv\|_{L^2}^2= \int \sum_{i=1}^{n-1} \xi_i^{2l}|\mc{F}_hv(\xi)|^2d\xi=\sum_{i=1}^{n-1}\|\mc{F}_h(hD_{x_i}^lv)\|_{L^2}^2.
\end{gather*} 
Using this in~\eqref{e:sobolevEst} {together with the fact that by Parseval's theorem for $u\in L^2$, $\|\mc{F}_hu\|_{L^2}=(2\pi h)^{\frac{n-1}{2}}\|u\|_{L^2}$} proves the Lemma.
\end{proof}

\begin{lem}
\label{l:octopus}
There exists $C_n>0$ depending only on $n$, $\delta_1>0$ {and $r_0>0$} so that for $0<\delta<\delta_1$ if $(x_0,\xi_{{j}})\in \Sigma_{x_0}$, $0<r<r_0$ and $\chi_j\in C_c^\infty(T^*M)$ with 
\begin{equation*}
\supp \chi_j \cap \Lambda_{\red{x_0}} \subset T_{\red{4\delta}}(\xi_{{j}},r),\qquad H_p\chi _j\equiv 0,\text{ on }\Lambda_{x_0,3\delta}
\end{equation*}
where $T_{\red{4\delta}}(\xi_{{j}},r)$ is as in~\eqref{e:tubes}.
 Then
\begin{equation}
\label{e:optimized}
\limsup_{h\to 0}h^{n-1}\|\chi_j u_\Sigma\|_{L^\infty(B(x_0,r(h)))}^2\leq C_{n}\delta^{-1} |\partial_{\xi}p({x_0},{\xi_j})|_g^{-1}{r^{n-1}}\int_{\Lambda_{x_0,3\delta}}\chi_j^2d\mu.
\end{equation}
\end{lem}

\begin{proof}
Let $a_{j,i}(x_1)$, $i=2,\dots n$ so that $\xi_i-a_{j,i}(x_1)$ vanishes on the bicharacteristic emanating from $({x_0},\xi_j)$. This is possible since we have chosen coordinates so that $\partial_{\xi_1}p(x_0,\xi_j)\neq 0$ and hence a bicharacteristic may be written locally as
$$
\gamma=\{(x,\xi)\mid x_1\in (-3\delta,3\delta),\,x'=x'(x_1),\,\xi=a(x_1)\}.
$$

Let $2l>n-1$ and $q_{j,i}=(\xi_i-a_i(x_1))^l$. Then, using $q=q_{j,i}$ in~\eqref{UPSHOT 1} gives
\begin{multline*} 
\|(hD_{x_i}-a_i(x_1))^l \chi_{\e,x_0}\chi_j(x,hD)u_{\Sigma}(x_1, \cdot) \|_{L^2_{x'}(\R^{n-1})}\leq 2\delta_0^{-1/2}\|b_{\e,x_0}(x,hD)q_{j,i}(x,hD)\chi_j(x,hD)u_{\Sigma}\|_{L^2_x(\R^n)}\\
+{C_0}\delta_0^{1/2}h^{-1}\|b_{\e,x_0}(x,hD)Pq_{j,i}(x,hD)\chi_j(x,hD)u_{\Sigma}\|_{L^2_x(\R^n)}+o_{\e,\delta}(1)
\end{multline*}
where $|\partial_\xi p|_g=|\partial_\xi p\cdot \partial_x|_g.$
Next, $q=1$ in~\eqref{UPSHOT 1} gives
$$
\|\chi_{\e,x_0}\chi_j u_{\Sigma}\|_{L_{x'}^2}\leq 2\delta_0^{-1/2}\|b_{\e,x_0}(x,hD)\chi_j(x,hD)u_{\Sigma}\|_{L^2_x(\R^n)}
+{C_0}\delta_0^{1/2}h^{-1}\|b_{\e,x_0}(x,hD)P\chi_j u_{\Sigma}\|_{L^2_x(\R^n)}+o_{\e,\delta}(1).
$$
Therefore, letting $w=e^{-i\langle x',a_j(x_1)\rangle/h}\chi_{\e,x_0} \chi_j u$ with $a_j(x_1)=(a_{j,2}(x_1),\dots,a_{j,n}(x_1))$ we see that 
$$
\|(hD_{x_i})^lw\|_{L^2_{x'}}\leq 2\delta_0^{-1/2}\|b_{\e,x_0}q_{j,i}\chi_ju_{\Sigma}\|_{L^2_x(\R^n)}
+{C_0}\delta_0^{1/2}h^{-1}\|b_{\e,x_0}Pq_{j,i}\chi_ju_{\Sigma}\|_{L^2_x(\R^n)}+o_{\e,\delta}(1)
$$
and 
$$
\|w\|_{L_{x'}^2}\leq 2\delta_0^{-1/2}\|b_{\e,x_0}\chi_ju_{\Sigma}\|_{L^2_x(\R^n)}+{C_0}\delta_0^{1/2}h^{-1}\|b_{\e,x_0}P\chi_j u_{\Sigma}\|_{L^2_x(\R^n)}+o_{\e,\delta}(1).
$$

Applying Lemma~\ref{l:improvedLinfty} to $w$ (with $\e=\alpha$) and using the fact that $\|w\|_{L^\infty}=\|\chi_{\e,x_0}\chi_ju_{\Sigma}\|_{L^\infty}$ gives for any $\alpha>0$ and $r(h)=o(1)$
\begin{multline*}
\limsup_{h\to 0}h^{n-1}\|\chi_j u_{\Sigma}\|^2_{L^\infty(B(x_0,r(h))}\leq C_{n,l}\alpha^{n-1}\left(\limsup_{h\to 0}\Big[{4}\delta_0^{-1}\|b_{\e,x_0}\chi_ju_\Sigma\|_{L^2_{x}}^2+{C_0^2}\delta_0 h^{-2}\|b_{\e,x_0}P\chi_ju_{\Sigma}\|_{L^2_{x}}^2\Big]\right)\\
\qquad +C_{n,l}\alpha^{n-2l-1}\left(\sum_{i=2}^{n}\limsup_{h\to 0}\Big[{4}\delta_0^{-1}\|b_{\e,x_0}q_{j,i}\chi_ju_\Sigma\|_{L^2_{x}}^2+{C_0^2}\delta_0 h^{-2}\|b_{\e,x_0}Pq_{j,i}\chi_ju_{\Sigma}\|_{L^2_{x}}^2\Big]\right)
\end{multline*}

In particular, applying Lemma~\ref{vanish}, 
\begin{align*}
\limsup_{h\to 0}h^{n-1}\|\chi_j u_\Sigma\|_{L^\infty(B(x_0,r(h)))}^2&\leq C_{n,l}\alpha^{n-1} \int  b_{\e,x_0}^2({4}\delta_0^{-1}\chi_j^2+{C_0^2}\delta_0|H_p\chi_j|^2)d\mu\\
&\qquad +C_{n,l}\alpha^{n-2l-1}\sum_{i=2}^{n}\int b_{\e,x_0}^2 ({4}\delta_0^{-1}\chi_j^2q_{j,i}^2+{C_0^2}\delta_0|H_p\chi_jq_{i,j}|^2)d\mu.
\end{align*}
Observe that \red{by~\eqref{e:platypus1} together with $0\leq b_{\e,x_0}^2\leq 1$, we have }
$$
\lim_{\e\to 0} b_{\e,x_0}^2\leq 1_{\Lambda_{x_0,3\delta}}.
$$
Sending $\e\to 0$ and using $H_p\chi_j=0$ on $\Lambda_{x_0,3\delta}$ (together with $\mu(T^*M)=1$ to apply the dominated convergence theorem) we have
\begin{equation}
\label{e:prelimEsta}
\begin{aligned}
\limsup_{h\to 0}h^{n-1}\|\chi_j u_\Sigma\|_{L^\infty(B(x_0,r(h)))}^2&\leq C_{n,l}{4}\delta_0^{-1}\alpha^{n-1} \int_{\Lambda_{x_0,3\delta}}  \chi_j^2d\mu\\
&\qquad +C_{n,l}\alpha^{n-2l-1}\sum_{i=2}^{n}\int_{\Lambda_{x_0,3\delta}} \chi_j^2({4}\delta_0^{-1}q_{j,i}^2+{C_0^2}\delta_0|H_pq_{i,j}|^2)d\mu
\end{aligned}
\end{equation}

Now, $\chi_j$ is supported on $T_{\red{4\delta}}(\xi,r)$ (see~\eqref{e:tubes}). Letting $\gamma$ be the bicharacteristic through $(x,\xi)$, we have by~\eqref{e:flow-1} that {for $\delta<\red{\frac{\delta_{M,p}}{3}}$ and $r<C_1^{-1}$ small enough}
$$
\sup\{d((x,\xi_1),\gamma)\mid (x,\xi_1)\in T_{\red{4\delta}}(\xi,r)\cap \Lambda_{x_0,3\delta}\}\leq 3r.
$$
Hence, since {$H_p(\xi_i-a_i(x_1))\equiv 0$ on $\gamma$, $H_pq_{j,i}=l(\xi_i-a_i(x_1))^{l-1}H_p(\xi_i-a_i(x_1))$ vanishes to order $l$ on $\gamma$ and there exists $C_2>0$ so that }
$$
\sup_{T_{\red{4\delta}}(\xi,r)\cap \Lambda_{x_0,3\delta}}|H_pq_{j,i}|\leq {C_2}r^l.
$$
Furthermore, by~\eqref{e:flow-1} {for $\red{\eta}\in B(\xi,r)$, and~\eqref{e:flow-2}
$$
|\xi_i(G_t(x_0,\red{\eta}))-\xi_i(G_t(x_0,\xi))|\leq 2r+C_1r^2,\qquad |a_i(x_1(G_t(x_0,\red{\eta})))-a_i(x_1(G_t(x_0,\xi)))|\leq C_3C_1\delta r.
$$}
Therefore,
$$
\sup_{T_{\red{4\delta}}(\xi,r)\cap \Lambda_{x_0,3\delta}}|q_{j,i}|\leq r^l(2+{C_1C_3}\delta+{C_1}r)^l
$$
{In particular, 
$$
\chi_j^2({4}\delta_0^{-1}q_{j,i}^2+{C_0^2}\delta_0|H_pq_{i,j}|^2)\leq r^{2l}\Big[4\delta_0^{-1}\big(2+C_1C_3\delta+C_1r\big)^{2l}+\delta_0C_0^2C_2^2\Big]\chi_j^2
$$
Thus, letting $\delta_1<{\min(C_1^{-1}C_3^{-1},(C_0C_2\sup|\partial_{\xi}p|_g)^{-1},\tilde{\delta})}$ and ${r_0<C_1^{-1}}$  we obtain 
$$
\chi_j^2({4}\delta_0^{-1}q_{j,i}^2+{C_0^2}\delta_0|H_pq_{i,j}|^2)\leq \delta_0^{-1}r^{2l}\big(1+4^{2l+1}\big)\chi_j^2.
$$
Using this in}~\eqref{e:prelimEsta} that
\begin{align*}
\limsup_{h\to 0}h^{n-1}\|\chi_j u_\Sigma\|_{L^\infty(B(x_0,r(h)))}^2&\leq C_{n,l}\delta_0^{-1}\int_{\Lambda_{x_0,3\delta}}\chi_j^2( {4}\alpha^{n-1} +\alpha^{n-2l-1}{(n-1)(4^{2l+1} +1)}r^{2l})d\mu.
\end{align*}
{Choosing $\alpha =r$ }and fixing $l=n$ gives~\eqref{e:optimized}.
\end{proof}

We now find an appropriate cover of $\Lambda_{x_0}$ that is adapted to $\mu_{\red{x_0}}$. 

\subsection{Decomposition of $\Lambda_{x_0}$}

\label{s:proof}

\begin{proof}[Proof of Theorem~\ref{thm:local}]
Recall that 
$$
\mu_{x_0}=\rho_{x_0} +fd\mc{H}^n_{x_0}
$$
 where $\rho_{x_0}\perp \mc{H}^n_{x_0}$ and $\mu_{x_0}$ is invariant under $G_t$. Therefore, by Lemma~\ref{l:invariance}, $\rho_{x_0}$ and $fd\mc{H}^n_{x_0}$ are invariant under $G_t$. 

Fix $0<\e \ll \delta$ arbitrary. By Lemma~\ref{l:tubes}, \red{with $\delta$ replaced by $4\delta$} there exist $((x_0,\xi_j),r_j)\in \Sigma_{x_0}\times \R_+$ satisfying~\eqref{e:mucond}. Let $K$ be large enough so that 
\begin{equation}
\label{e:foot}
\rho_{x_0}\left(\Lambda_{x_0,\red{4\delta}}\setminus \bigcup_{j=1}^K T_{\red{4\delta}}(\xi_j,r_j)\right)<\e.
\end{equation}
Let $\chi_j\in C_c^\infty(T^*M;[0,1])$ satisfy~\eqref{e:partition} for $((x_0,\xi_j), r_j)$ $j=1,\dots K$.

Define $\psi=1-\sum\chi_j$. Applying Lemma~\ref{l:octopus} (with $\xi=\xi_j$, $r=r_j$, $\chi=\chi_j$), summing and using the triangle inequality, we have
\begin{equation}
\label{e:singularPart}
\begin{aligned}
\limsup_{h\to 0}h^{\frac{n-1}{2}}\|(1-\psi(x,hD)) u_{\Sigma}\|_{L^\infty(B(x_0,r(h)))}&\leq C_{n,\delta} \sum_{j=1}^K r_j^{(n-1)/2}\left(\int_{\Lambda_{x_0}}\chi_j^2d\mu\right)^{1/2}\\
&\leq C_{n,\delta} \left(\sum_j r_j^{n-1}\right)^{1/2}\left(\int _{\Lambda_{x_0}}\sum_j\chi_j^2d\mu\right)^{1/2}\\
&\leq C_{n,\delta}\e^{1/2}\mu(\Lambda_{x_0,\red{4\delta}})
\end{aligned} 
\end{equation}
where in the last line we use $0\leq \chi_j\leq 1$ and $0\leq \sum \chi_j\leq 1$. 

Next we estimate  $\psi(x,hD)u_\Sigma$.  By the Besicovitch--Federer Covering Lemma \cite[Theorem 1.14, Example (c)]{Hein01}, there exists a constant $C_n$ depending only on $n$ and $\gamma_0=\gamma_0(\Sigma_{x_0})$ so that for all $0<\gamma<\gamma_0$, there exists $\xi_1,\dots \xi_{N(\gamma)}$ with $N(\gamma)\leq C\gamma^{1-n}$ so that 
$$
\Sigma_{x_0}\subset \bigcup_{j=1}^{N(\gamma)}B(\xi_k,\gamma)
$$ 
and  each point in $\Sigma_{x_0}$ lies in at most $C_n$ balls $B(\xi_k,\gamma)$. Let $\psi_k$,  $k=1,\dots N(\gamma)$ satisfy~\eqref{e:partition},~\eqref{e:partition2} (with $\xi_j=\xi_k$, $2r_j=\gamma$, and $K=N(\gamma)$).  Observe that applying Lemma~\ref{l:octopus} (with $\xi=\xi_k$, $r=\gamma$, and $\chi_j=\psi_k\psi$),
\begin{align*} 
\limsup_{h\to 0}h^{n-1}\|\psi(x,hD)\psi_k(x,hD) u_\Sigma\|^2_{L^\infty(B(x_0,r(h)))}&\leq C_{n}\delta^{-1}|\partial_{\xi}p(x_0,\xi_k)|_g^{-1}\int_{\Lambda_{x_0,3\delta}}\psi_k^2\psi^2\gamma^{n-1}d\mu
\end{align*}

Notice that 
$$
\sum_k\psi \psi_k \equiv 1 \text{ on } \Lambda_{x_0,3\delta}
$$
and therefore Lemma~\ref{l:flowoutNeed} implies 
$$
\limsup_{h\to 0}h^{\frac{n-1}{2}}\Big\|\psi(x,hD)\Big[1-\sum_{k}\psi_k(x,hD)\Big]u_\Sigma\Big\|_{L^\infty(B(x_0,r(h)))}=0.
$$
So, applying the triangle inequality,
\begin{align*} 
&\limsup_{h\to 0}h^{\frac{n-1}{2}}\Big\|\psi(x,hD)u_\Sigma\Big\|_{L^\infty(B(x_0,r(h)))}\\
&\quad\leq C_{n,\delta}\sum_k\left(\int_{\Lambda_{x_0,3\delta}}\psi_k^2\psi^2\gamma^{n-1}d\rho_{x_0}\right)^{1/2}+C_n\delta^{-1/2}\sum_k\left(\int_{\Lambda_{x_0,3\delta}}|\partial_{\xi}p(x_0,\xi_k)|_g^{-1}\psi_k^2\psi^2\gamma^{n-1}fd\mc{H}^n_{x_0}\right)^{1/2}\\
&\quad=:C_{n,\delta}I+II
\end{align*} 
Use~\eqref{e:foot} to estimate
\begin{align*} 
I&\leq C\gamma^{\frac{n-1}{2}}N(\gamma)^{1/2}\left(\int_{\Lambda_{x_0,3\delta}}\sum_k\psi_k^2\psi^2d\rho_{x_0}\right)^{1/2}\leq C\rho_{x_0}\left(\Lambda_{x_0,3\delta}\setminus \bigcup_{j=1}^KT_{\red{4\delta}}(\xi_j,r_j)\right)^{1/2}\leq C\e^{1/2}.
\end{align*}

{Now, using that $fd\mc{H}^n_{\red{x_0}}$ is $G_t$ invariant and applying Lemma~\ref{l:inv}
\begin{align*}
II&\leq C_n\sum_k\left(\int_{\Sigma_{x_0}}|\partial_{\xi}p(x_0,\xi_k)|_g^{-1}\psi_k^2\psi^2\gamma^{n-1}f(0,q)|\nu(H_p)|(0,q)d\vol_{\Sigma_{x_0}}\right)^{1/2}.
\end{align*}
Define for $0\leq \theta \in L^1(\vol_{\Sigma_{x_0}})$
$$
T_\gamma \theta:=C_n\sum_k\left(\int_{\Lambda_{x_0,3\delta}}|\partial_{\xi}p(x_0,\xi_k)|_g^{-1}\psi_k^2\psi^2\gamma^{n-1}\theta(0,q)|\nu(H_p)|(0,q)d\vol_{\Sigma_{x_0}}\right)^{1/2}.
$$
Then,
\begin{align*}
T_\gamma \theta&\leq C_nN(\gamma)^{\frac{1}{2}}\gamma^{\frac{n-1}{2}}\left(\int \sum_k|\partial_{\xi}p(x_0,\xi_k)|_g^{-1}\psi_k^2\psi^2\theta(0,q)|\nu(H_p)|(0,q)d\vol_{\Sigma_{x_0}}\right)^{1/2}
&\leq C\|\theta\|_{L^1}^{1/2}
\end{align*}
where $C$ is independent of $\gamma$. 

Now, suppose that $\theta\geq 0$ is continuous. For $\gamma$ small enough, $C_n^{-1} \gamma^{n-1}\leq \vol_{\Sigma_{\red{x_0}}}(B(\xi_k,\gamma))\leq C_n \gamma^{n-1}$, where $C_n$ depends only on $n$,
\red{To see this, recall that on any compact Riemannian manifold $\tilde{M}$ of dimension $n-1$, there is $C_n$ depending only on $n$ and $C>0$ depending on $\tilde{M}$ so that for all $q\in \tilde{M}$}
$$
\red{|\vol_{\tilde{M}}(B(q,\gamma))-C_n\gamma^{n-1}|\leq C\gamma^{n}.}
$$
\red{This follows from computing in geodesic normal coordinates. In fact, $C$ depends only on bounds on the curvature of $\tilde{M}$.}

\red{Using this, we have}
\begin{align*} 
T_\gamma \theta&\leq C_n\int_{\Sigma_{x_0}}\sum_k 1_{B(\xi_k,\gamma)}\left(\frac{1}{|\partial_{\xi}p(x_0,\xi_k)|_g\vol_{\Sigma_{x_0}}(B(\xi_k,\gamma))}\int_{B(\xi_k,\gamma)} \theta(0,q)|\nu(H_p)|(0,q)d\vol_{\Sigma_{x_0}}\right)^{1/2}d\vol_{\Sigma_{x_0}}
\end{align*}
Now, as discussed in Remark~\ref{r:compact}, we may assume $\Sigma_{x_0}$ is compact. Then, since \red{$\theta$, the metric, $g$, and $p$ are} continuous, \red{they are} uniformly continuous. In particular, for any $\e_0>0$ there exists $\gamma$ small enough so that for all $\xi\in \Sigma_{x_0}$ and $q_0\in B(\xi,\gamma)$, 
$$
\left(\frac{1}{|\partial_{\xi}p(x_0,\xi)|_g\vol_{\Sigma_{x_0}}(B(\xi,\gamma))}\int_{B(\xi,\gamma)} \theta(0,q)|\nu(H_p)|(0,q)d\vol_{\Sigma_{x_0}}\right)^{1/2}\leq \sqrt{\frac{\theta|\nu(H_p)|}{|\partial_{\xi}p|_g}}(q_0)+\frac{\e_0}{\vol_{\Sigma_{x_0}}}.
$$
Thus, 
$$T_\gamma \theta\leq C_n\int \sqrt{\frac{\red{\theta}|\nu(H_p)|}{|\partial_{\xi}p|_g}}d\vol_{\Sigma_{x_0}}+C_n\e_0.$$
}

{
Next, let $\theta_m\geq 0$ continuous with $\theta_m\to f$ in $L^1$. We may assume by taking a subsequence that $\theta_m\to f$ a.e. Fix $\e_0>0$. Then since $\sqrt{a+b}\leq \sqrt{a}+\sqrt{b}$,
$$
|T_\gamma f|\leq T_\gamma|f-\theta_m|+T_\gamma \theta_m\leq C\|f-\theta_m\|_{L^1}^{1/2}+T_\gamma \theta_m.
$$
For $m\geq M$, $C\|f-\theta_m\|_{L^1}^{1/2}\leq \e_0$ and hence
$$
|T_\gamma f|\leq \e_0+T_\gamma \theta_m.
$$
Now, 
\begin{align*}
\int \sqrt{\frac{\theta_m|\nu(H_p)|}{|\partial_{\xi}p|_g}}d\vol_{\Sigma_{x_0}}&=\int \Big[\sqrt{\frac{\max(\theta_m,1)|\nu(H_p)|}{|\partial_{\xi}p|_g}}\Big]d\vol_{\Sigma_{x_0}}\\
&\qquad\qquad+\int 1_{\theta_m\geq 1}\Big[\sqrt{\frac{\theta_m|\nu(H_p)|}{|\partial_{\xi}p|_g}}-\sqrt{\frac{|\nu(H_p)|}{|\partial_{\xi}p|_g}}\Big]d\vol_{\Sigma_{x_0}}.
\end{align*}
Observe next that $\max(\theta_m,1)\to \max(f,1)$ a.e. and by the dominated convergence theorem, 
$$\int \Big[\sqrt{\frac{\max(\theta_m,1)|\nu(H_p)|}{|\partial_{\xi}p|_g}}\Big]d\vol_{\Sigma_{x_0}}\to \int \Big[\sqrt{\frac{\max(f,1)|\nu(H_p)|}{|\partial_{\xi}p|_g}}\Big]d\vol_{\Sigma_{x_0}}.$$
}

{
Next, 
\begin{align*}
\Big|\int 1_{\theta_m\geq  1}\Big[\sqrt{\frac{\theta_m|\nu(H_p)|}{|\partial_{\xi}p|_g}}-\sqrt{\frac{f|\nu(H_p)|}{|\partial_{\xi}p|_g}}\Big]d\vol_{\Sigma_{x_0}}\Big|&=\Big|\int 1_{\theta_m\geq 1}\frac{(\theta_m-f)\sqrt{|\nu(H_p)|}}{\sqrt{|\partial_{\xi}p|_g}(\sqrt{\theta_m}+\sqrt{f})}d\vol_{\Sigma_{x_0}}\Big|\\
&\leq C\|\theta_m-f\|_{L^1}.
\end{align*}
In particular, this proves that 
$$
\int \sqrt{\frac{\theta_m|\nu(H_p)|}{|\partial_{\xi}p|_g}}d\vol_{\Sigma_{x_0}}\to \int \sqrt{\frac{f|\nu(H_p)|}{|\partial_{\xi}p|_g}}d\vol_{\Sigma_{x_0}}.
$$
Therefore, for $m\geq M_1$,
$$
\int \sqrt{\frac{\theta_m|\nu(H_p)|}{|\partial_{\xi}p|_g}}d\vol_{\Sigma_{x_0}}\leq \int \sqrt{\frac{f|\nu(H_p)|}{|\partial_{\xi}p|_g}}d\vol_{\Sigma_{x_0}}+\e_0.
$$ 
Let $m\geq \max(M,M_1)$ and choose $\gamma$ small enough so that 
$$
T_\gamma \theta_m\leq C_n\int \sqrt{\frac{\theta_m|\nu(H_p)|}{|\partial_{\xi}p|_g}}d\vol_{\Sigma_{x_0}}+\e_0.
$$
Then,
$$
T_\gamma f\leq \e_0 +T_\gamma \theta_m\leq 2\e_0+ C_n\int \sqrt{\frac{f|\nu(H_p)|}{|\partial_{\xi}p|_g}}d\vol_{\Sigma_{x_0}} + C_n\e_0.
$$  
In particular, 
$$
\limsup_{\gamma\to 0}II=\limsup_{\gamma\to 0}T_\gamma f\leq C_n\int \sqrt{\frac{f|\nu(H_p)|}{|\partial_{\xi}p|_g}}d\vol_{\Sigma_{x_0}}.
$$
Therefore, sending $h\to 0$ then $\gamma \to 0$ and finally $\e\to 0$ proves the theorem.
} 

\end{proof}

\section{Construction of Modes - Proof of Theorem~\ref{thm:modes}}

\begin{proof}[Proof of Theorem~\ref{thm:modes}]
We apply the construction in \cite[Lemma 7]{SoggeTothZelditch}. Let $p=\frac{1}{2}(|\xi|_g^2-1)$ and $G_t=\exp(tH_p)$ so that $G_t|_{S^*M}$ is the unit speed geodesic flow. Let $g_1\in L^2(S^*_{\red{z_0}}M)$ have $|g_1|^2=f|_{S^*_{\red{z_0}}M}$ and $g_{1,\e}\in C^\infty(S^*_{\red{z_0}}M)$ have $\|g_{1,\e} - g_1\|_{L^2(S^*_{\red{z_0}}M)}<\e$. For $A\subset S^*_{\red{z_0}}M$ Borel, define the measure
$$
\tilde{\rho}_{\red{z_0}}(A)=\frac{1}{2\delta}\rho_{\red{z_0}}\left(\bigcup_{t=-\delta}^\delta G_t(A)\right).
$$
Let $g_{2,\e}\in C^\infty(S^*_{\red{z_0}}M)$ have $|g_{2,\e}|^2dS_\phi\underset{{\e\to 0}}{\longrightarrow} \tilde{\rho}_{\red{z_0}}$ as a measure where $S_\phi$ is the surface measure on $S^{n-1}$. Finally, define $g_\e=g_{1,\e}+g_{2,\e}$.

{We apply the arguments \cite[Section 2.3.1]{SoggeTothZelditch} to see that there exists $\Phi_{\e,j}$ such that}
$$
\| (-h_j^2\Delta_g-1)\Phi_{\e,j}\|_{L^2}=O_\e(h_j^2),\qquad C+O_{\e}(h_j)\geq \|\Phi_{\e,j}\|_{L^2}\geq c+O_{\e}(h_j).
$$
{and, in} normal geodesic coordinates at $\red{z_0}$, we have 
$$
\Phi_{\e,j}(x)=(2\pi h_j)^{\frac{1-n}{2}}\int e^{i\left\langle x,\frac{\theta}{|\theta|}\right\rangle/h_j}g_{\e}\left(\frac{\theta}{|\theta|}\right)\chi_R(|\theta|)d\theta,
$$
where $\chi_R\in C_c^\infty((0,\infty);[0,1])$ with $\chi_R\equiv 1$ on $[ 1,R]$, $\supp \chi_R\subset (0,2R)$ and 
\begin{equation}
\label{e:kangaroo}
\int \chi_R(\alpha)\alpha^{n-1}d\alpha=1.
\end{equation}
{The remainder of the proof consists of analyzing this oscillatory integral.}

Choose $\e_j\to 0$ so slowly that 
\begin{gather*} 
\lim_{j\to \infty}\| (-h_j^2\Delta_g-1^2)\Phi_{\e_j,j}\|_{L^2}h_j^{-1}\to 0,\qquad 2C\geq \limsup_{j\to \infty}\|\Phi_{\e_j,j}\|_{L^2}\geq \liminf_{j\to \infty}\|\Phi_{\e_j,j}\|_{L^2}>c/2.
 \end{gather*}
Then,
$$
\|(-h_j^2\Delta_g-1)\Phi_{\e_j,j}\|_{L^2}=o(h_j\|\Phi_{\e_j,j}\|_{L^2}).
$$

Fix $N>0$ to be chosen large and $\e_j\to 0$ slowly enough so that 
\begin{equation}
\label{e:derEst}
\sup_{|\alpha|\leq N}\sup_{S^*_{\red{z_0}}M}|\partial^{|\alpha|}g_{\e_j}|h_j\to 0.
\end{equation}

Under this condition, we compute the defect measure of $\Phi_{\e_j,j}$. {Note that since $\|\Phi_{\e_j,j}\|$ is uniformly bounded, we may assume by taking subsequences that the defect measure exists.} Let $b\in C_c^\infty(T^*M)$ supported in 
$$
A_{\delta}:=\{x\mid \delta\leq |{d_M}(\red{z_0},x)|\leq 2\delta\}
$$
{where $d_M$ is the distance on $M$.}
Then, letting $\psi\in C_c^\infty(\R \setminus\{0\})$ have $\psi\equiv 1$ on $[\delta,2\delta]$, 
$$
b(x,h_jD)\Phi_{\e_j,j}=(2\pi h_j)^{\frac{1-3n}{2}}\int e^{i\left(\langle x-y,\xi\rangle +\left\langle y,\frac{\theta}{|\theta|}\right\rangle\right)/h_j}b(x,\xi)\psi(|y|)g_{\e_j}\left(\frac{\theta}{|\theta|}\right)\chi_R(|\theta|)d\theta dyd\xi+O_{L^2}(h_j^\infty).
$$
Performing stationary phase in the $(y,\xi)$ variables gives
$$
b(x,h_jD)\Phi_{\e_j,j}=(2\pi h_j)^{\frac{1-n}{2}}\int e^{i\left\langle x,\frac{\theta}{|\theta|}\right\rangle/h_j}\left[b\left(x,\frac{\theta}{|\theta|}\right)+h_j e(x,\theta)\right]g_{\e_j}\left(\frac{\theta}{|\theta|}\right)\chi_R(|\theta|)d\theta +O_{L^2}(h_j^\infty)
$$
where $e\in C^\infty(\re^{2n})$ has $\supp r\subset \supp b$ and is independent of $\e$. 
\begin{multline*} 
\langle b(x,h_jD)\Phi_{\e_j,j},\Phi_{\e_j,j}\rangle=\\
 (2\pi h_j)^{1-n}\int_{A_\delta}\int e^{i\frac{|x|}{h_j}\left\langle \frac{x}{|x|}, \frac{\theta}{|\theta|}-\frac{\omega}{|\omega|}\right\rangle} g_{\e_j}\left(\frac{\theta}{|\theta|}\right)\left[b\left(x,\frac{\theta}{|\theta|}\right)+h_je(x,\theta)\right]\overline{g_{\e_j}\left(\frac{\omega}{|\omega|}\right)}\chi_R(|\theta|)\chi_R(|\omega|)d\theta d\omega dx \\+O(h_j^\infty).
\end{multline*}
We write the integral in polar coordinates $x=r\phi,$ $\theta=\alpha \Theta$, and $\omega=\beta\Omega$. Since $|r|>\delta$ on $A_\delta$, we {may} perform stationary phase in $\Omega$ and $\Theta$. Using~\eqref{e:derEst} with $N>n+2$ together with the remainder estimate \cite[Theorem 3.16]{EZB} to control the error uniformly as $j\to\infty$, gives
\begin{multline*} 
\int_{S^{n-1}}\int_{\mathbb{R}^3_+} [ |g_{\e_j}(\phi)|^2b(r\phi,\phi)+|g_{\e_j}(-\phi)|^2b(r\phi,-\phi)\\
+c_1e^{2ir/h}g_{\e_j}(\phi)\overline{g_{\e_j}(-\phi})b(r\phi,\phi)+c_2e^{-2ir/h}g_{\e_j}(-\phi)\overline{g_{\e_j}(\phi})b(r\phi,-\phi)]\alpha^{n-1}\beta^{n-1}\\
\chi_R(\alpha)\chi_R(\beta)\psi(r)d\alpha d\beta dr dS_\phi +o(1)
\end{multline*}
Integration by parts in $r$ then shows that the second two terms are lower order and yields
\begin{equation*} 
\int_{S^{n-1}}\int_{\mathbb{R}^3_+} [ |g_{\e_j}(\phi)|^2b(r\phi,\phi)+|g_{\e_j}(-\phi)|^2b(r\phi,-\phi)]\alpha^{n-1}\beta^{n-1}\chi_R(\alpha)\chi_R(\beta)d\alpha d\beta dr dS_\phi +o(1)
\end{equation*}
Sending $j\to \infty$ gives 
$$
\left(\int_0^\infty \chi_R(\alpha )\alpha^{n-1}d\alpha\right)^2\int_{\R}\int_{S^{n-1}}b(r\phi,\phi)(d\tilde{\rho}_{\red{z_0}}(\phi) +|g_1|^2dS_\phi)dr=\int_{\Lambda_{z_0}}b(x,\xi)(d\rho_{\red{z_0}}+fd\textup{Vol}_{\Lambda_{\red{z_0}}})
$$
where we use~\eqref{e:kangaroo}. 

Using that the defect measure of $\Phi_{\e_j,j}$ is invariant under $G_t$ then shows that $\Phi_{\e_j,j}$ has defect measure 
$$
\mu= d\rho_{\red{z_0}} +fd\vol_{\Lambda_{\red{z_0}}}.
$$
{This implies that for $\chi \in C_c^\infty(T^*M)$ with $\chi \equiv 1$ on $|\xi|_g\leq 2$, 
$$
\langle \chi(x,h_jD)\Phi_{\e_j,j},\Phi_{\e_j,j}\rangle_{L^2(M)}\to 1.
$$
Now, since $(-h_j^2\Delta_g-1)\Phi_{\e_j,j}=o_{L^2}\red{(h_j)}$, an elliptic parametrix construction as in~\eqref{e:elliptic} implies $(I-\chi(x,h_jD))\Phi_{\e_j,j}=o_{L^2(M)}(h_j)$ and 
and hence $\|\Phi_{\e_j,j}\|_{L^2}\to 1$. }

{Next observe that}
$$
\Phi_{\e_j,j}(\red{z_0})=(2\pi h_j)^{\frac{1-n}{2}}\int_{\R^n}g_{\e_j}\left(\frac{\theta}{|\theta|}\right)\chi_R(|\theta|)d\theta=(2\pi h_j)^{\frac{1-n}{2}}\int_{S^{n-1}}(g_{1,\e_j}(\phi)+g_{2,\e_j}(\phi))dS_\phi.
$$
Since $\tilde{\rho}_{\red{z_0}}\perp d\vol_{\Sigma_{\red{z_0}}}$ and $|g_{2,\e_j}|^2dS_\phi\to \tilde{\rho}_{\red{z_0}}$ as a measure, for any $\delta>0$, there exists $A\subset S^{n-1}$ so that 
$$
\int_{A^c}|g_{2,\e_j}|^2dS_\phi \to 0,\qquad \int_A dS_\phi<\delta.
$$
 Therefore, 
$$
\left|\int_{S^{n-1}}g_{2,\e_j}(\phi)dS_\phi\right|\leq C\left(\int_{A^c}|g_{2,\e_j}|^2dS_\phi\right)^{1/2} +\left(\int_{S^{n-1}}|g_{2,\e_j}|^2dS_\phi\right)^{1/2}\delta^{1/2}
$$
so, for all $\delta>0$, 
$$
\limsup_{j\to \infty}\left|\int_{S^{n-1}}g_{2,\e_j}(\phi)dS_\phi\right|\leq C\delta^{1/2}.
$$
In particular, 
$$
\lim_{j\to \infty}\int_{S^{n-1}}g_{2,\e_j}(\phi)dS_\phi =0.
$$
Finally, using that $g_{1,\e_j}\to g_1$ in $L^2$ and hence also in $L^1$
$$
\lim_{j\to \infty}u_j(\red{z_0})h_j^{\frac{n-1}{2}}=(2\pi)^{\frac{1-n}{2}}\int_{S^{n-1}} g_1(\phi)dS_\phi.
$$
Letting $u_j=\Phi_{\e_j,j}/\|\Phi_{\e_j,j}\|_{L^2}$ then proves the lemma. 
\end{proof}

\appendix
 \section{Semiclassical notation}\label{a:semiclassics}
 
{We next review the notation used for semiclassical operators and symbols and some of the basic properties. Recall that for a compact manifold $M$ of dimension $n$, we write
$$S^m(T^*M):=\{a(\,\cdot\,;h) \in C^\infty(T^*M):\;  |\partial_x^\alpha\partial_\xi^\beta a(x,\xi;h)|\leq C_{\alpha\beta}(1+|\xi|)^{m-|\beta|}\}$$
and $S^\infty(T^*M)=\bigcup_m S^m$. 
We write $\Psi^m(M)$ for the semiclassical pseudodifferential operators of order $m$ on $M$, $\Psi^\infty(M)=\bigcup_m\Psi^m(M)$ and 
$$ Op_h:S^m(T^*M)\to \Psi^m(M)$$
for a quantization procedure with $Op_h(1)=\Id+O_{\mathcal{D}'\to C^\infty}(h^\infty)$ and for $u$ supported in a coordinate patch, $\varphi\in C_c^\infty(M)$ with $\varphi\equiv 1$ on $\supp u$ we have
$$Op_h(a)u(x)=\frac{1}{(2\pi h)^n}\iint e^{\frac{i}{h}\langle x-y,\xi\rangle}\varphi(x)a(x,\xi)u(y)d\xi dy +O_{\mathcal{D'}\to C^\infty}(h^\infty)u.$$
We will often write $a(x,hD)$ for $Op_h(a)$.}

{
There exists a principal symbol map 
 $$
 \sigma:\Psi^m(M)\to S^m(T^*M)/hS^{m-1}(T^*M) 
 $$
 so that 
 \begin{gather*} 
 Op_h\circ \sigma (A)=A+O_{\Psi^{m-1}}(h),\quad A\in \Psi^m,\qquad\qquad\sigma \circ Op_h=\pi : S^m\to S^m/hS^{m-1},
 \end{gather*}
 where $\pi$ is the natural projection map. Moreover, for $A\in \Psi^{m_1},\,B\in\Psi^{m_2}$, \medskip
\begin{itemize} 
\item $\sigma(AB)=\sigma(A)\sigma(B)\;\; \in S^{m_1+m_2}/hS^{m_1+m_2-1},$ \medskip
\item  $\sigma([A,B])=\frac{h}{i}\big\{\sigma(A),\sigma(B)\big\}\;\;\in hS^{m_1+m_2-1}/h^2S^{m_1+m_2-2},$ \medskip
 \end{itemize}
 where $\{\cdot,\cdot\}$ denotes the poisson bracket. For more details on the semiclassical calculus see e.g. \cite[Chapters 4,14]{EZB} \cite[Appendix E]{ZwScat}. }


\bibliography{biblio}

\begin{thebibliography}{KTZ07}

\bibitem[Ava56]{Ava}
Vojislav~G. Avakumovi\'c.
\newblock \uppercase{\"u}ber die {E}igenfunktionen auf geschlossenen
  {R}iemannschen {M}annigfaltigkeiten.
\newblock {\em Math. Z.}, 65:327--344, 1956.

\bibitem[B{\'e}r77]{Berard77}
Pierre~H. B{\'e}rard.
\newblock On the wave equation on a compact {R}iemannian manifold without
  conjugate points.
\newblock {\em Math. Z.}, 155(3):249--276, 1977.

\bibitem[Bla10]{BlairSasaki}
David~E. Blair.
\newblock {\em Riemannian geometry of contact and symplectic manifolds}, volume
  203 of {\em Progress in Mathematics}.
\newblock Birkh\"auser Boston, Inc., Boston, MA, second edition, 2010.

\bibitem[BS02]{BrinStuck}
Michael Brin and Garrett Stuck.
\newblock {\em Introduction to dynamical systems}.
\newblock Cambridge University Press, Cambridge, 2002.

\bibitem[BS15]{BlSo15}
Matthew~D. Blair and Christopher~D. Sogge.
\newblock Refined and microlocal {K}akeya--{N}ikodym bounds for eigenfunctions
  in two dimensions.
\newblock {\em Anal. PDE}, 8(3):747--764, 2015.

\bibitem[BS17]{BlSo17}
Matthew~D. Blair and Christopher~D. Sogge.
\newblock Refined and microlocal {K}akeya--{N}ikodym bounds of eigenfunctions
  in higher dimensions.
\newblock {\em Communications in Mathematical Physics}, 356(2):501--533, 2017.

\bibitem[Don01]{Donnelly}
Harold Donnelly.
\newblock Bounds for eigenfunctions of the {L}aplacian on compact {R}iemannian
  manifolds.
\newblock {\em J. Funct. Anal.}, 187(1):247--261, 2001.

\bibitem[DZ17]{ZwScat}
Semyon Dyatlov and Maciej Zworski.
\newblock Mathematical theory of scattering resonances.
\newblock 2017.

\bibitem[GT17]{GT}
Jeffrey Galkowski and John~A Toth.
\newblock Eigenfunction scarring and improvements in ${L}^\infty$ bounds.
\newblock {\em Analysis \& PDE}, 11(3):801--812, 2017.

\bibitem[Hei01]{Hein01}
Juha Heinonen.
\newblock {\em Lectures on analysis on metric spaces}.
\newblock Universitext. Springer-Verlag, New York, 2001.

\bibitem[H{\"o}r68]{Ho68}
Lars H{\"o}rmander.
\newblock The spectral function of an elliptic operator.
\newblock {\em Acta Math.}, 121:193--218, 1968.

\bibitem[IS95]{I-s}
H.~Iwaniec and P.~Sarnak.
\newblock {${L}^\infty$} norms of eigenfunctions of arithmetic surfaces.
\newblock {\em Ann. of Math. (2)}, 141(2):301--320, 1995.

\bibitem[KTZ07]{KTZ}
Herbert Koch, Daniel Tataru, and Maciej Zworski.
\newblock Semiclassical {$L^p$} estimates.
\newblock {\em Ann. Henri Poincar{\'e}}, 8(5):885--916, 2007.

\bibitem[Lev52]{Lev}
B.~M. Levitan.
\newblock On the asymptotic behavior of the spectral function of a self-adjoint
  differential equation of the second order.
\newblock {\em Izvestiya Akad. Nauk SSSR. Ser. Mat.}, 16:325--352, 1952.

\bibitem[Saf88]{Saf88}
Yu.~G. Safarov.
\newblock Asymptotic of the spectral function of a positive elliptic operator
  without the nontrap condition.
\newblock {\em Functional Analysis and Its Applications}, 22(3):213--223, 1988.

\bibitem[Sog88]{So88}
Christopher~D Sogge.
\newblock Concerning the ${L}^p$ norm of spectral clusters for second-order
  elliptic operators on compact manifolds.
\newblock {\em Journal of functional analysis}, 77(1):123--138, 1988.

\bibitem[Sog93]{SoggeBook}
Christopher~D. Sogge.
\newblock {\em Fourier integrals in classical analysis}, volume 105 of {\em
  Cambridge Tracts in Mathematics}.
\newblock Cambridge University Press, Cambridge, 1993.

\bibitem[Sog11]{So11}
Christopher~D Sogge.
\newblock {K}akeya-{N}ikodym averages and $ {L}^p $-norms of eigenfunctions.
\newblock {\em Tohoku Mathematical Journal, Second Series}, 63(4):519--538,
  2011.

\bibitem[STZ11]{SoggeTothZelditch}
Christopher~D. Sogge, John~A. Toth, and Steve Zelditch.
\newblock About the blowup of quasimodes on {R}iemannian manifolds.
\newblock {\em J. Geom. Anal.}, 21(1):150--173, 2011.

\bibitem[SZ02]{SoggeZelditch}
Christopher~D. Sogge and Steve Zelditch.
\newblock Riemannian manifolds with maximal eigenfunction growth.
\newblock {\em Duke Math. J.}, 114(3):387--437, 2002.

\bibitem[SZ16a]{SZ16I}
Christopher~D. Sogge and Steve Zelditch.
\newblock Focal points and sup-norms of eigenfunctions.
\newblock {\em Rev. Mat. Iberoam.}, 32(3):971--994, 2016.

\bibitem[SZ16b]{SZ16II}
Christopher~D. Sogge and Steve Zelditch.
\newblock Focal points and sup-norms of eigenfunctions {II}: the
  two-dimensional case.
\newblock {\em Rev. Mat. Iberoam.}, 32(3):995--999, 2016.

\bibitem[TZ02]{TZ02}
John~A. Toth and Steve Zelditch.
\newblock Riemannian manifolds with uniformly bounded eigenfunctions.
\newblock {\em Duke Math. J.}, 111(1):97--132, 2002.

\bibitem[TZ03]{TZ03}
John~A. Toth and Steve Zelditch.
\newblock Norms of modes and quasi-modes revisited.
\newblock In {\em Harmonic analysis at {M}ount {H}olyoke ({S}outh {H}adley,
  {MA}, 2001)}, volume 320 of {\em Contemp. Math.}, pages 435--458. Amer. Math.
  Soc., Providence, RI, 2003.

\bibitem[Zwo12]{EZB}
Maciej Zworski.
\newblock {\em Semiclassical analysis}, volume 138 of {\em Graduate Studies in
  Mathematics}.
\newblock American Mathematical Society, Providence, RI, 2012.

\end{thebibliography}
\bibliographystyle{alpha}

\end{document}